\documentclass[11pt,twoside,a4paper]{article}

\usepackage{amsmath,amsfonts,euscript,amssymb,amsthm,a4,times}

\makeatletter \catcode`@=11
\newbox\tr@tto
\setbox\tr@tto=\hbox{{\count0=0\dimen0=-,9pt\dimen1=1,1pt\loop\ifnum
    \count0<11 \advance \count0 by1 \vrule width.51pt height\dimen1
    depth\dimen0\kern-0.17pt\advance\dimen0 by-0.05pt\advance\dimen1
    by0.1pt\repeat \loop\ifnum\count0<21\advance \count0 by1 \vrule
    width.6pt height\dimen1 depth\dimen0\kern-0.2pt \advance\dimen0
    by-0.1pt\advance\dimen1 by 0.05pt\repeat}}
\def\medint{\displaystyle\copy\tr@tto\kern-10.4pt\int}

 \allowdisplaybreaks

\def\R{\mathbb R}
\def\H{\mathcal H}

\def\Om{\Omega}
\def\N{\mathbb N}
\def\e{\varepsilon}
\def\ro{\varrho}

\newtheorem{Thm}{Theorem}[section]
\newtheorem{Lem}{Lemma}[section]

\numberwithin{equation}{section}

\title{Regularity Results for an Optimal Design Problem with a Volume Constraint}

\author{Menita Carozza, Irene Fonseca and Antonia Passarelli di Napoli}

\begin{document}

\maketitle

\begin{abstract}
\noindent Regularity results for minimal  configurations of variational problems involving both bulk and surface energies and subject to a volume constraint are established. The bulk energies are convex functions with $p$-power growth, but are otherwise not subjected to any further structure conditions. For  a minimal configuration $(u,E)$, H\"older continuity of the function $u$ is proved  as well as  partial regularity of the boundary of the minimal set $E$. Moreover, full regularity of the boundary of the minimal set is obtained under suitable closeness assumptions on the eigenvalues of the bulk energies.
\end{abstract}



\bigskip

\section{Introduction and statements}

\qquad In this paper we study minimal energy configurations of a mixture of two  materials in a bounded, connected open set $\Om\subset\R^n$, when the perimeter of the interface between the materials is penalized.
Precisely, the energy is given by
\begin{equation}\label{intro0}
{\mathcal I}(u,E):=\int_\Om \left(F(\nabla u)+\chi_{_{E}}G(\nabla u)\right)\,dx+P(E,\Om)\,,
\end{equation}
where $E\subset\Om$ is a  set of finite perimeter,  $u\in W^{1,p}(\Om)$, $p>1$, $\chi_E$ is the characteristic function of the set $E$ and $P(E,\Om)$ denotes the perimeter
 of $E$ in $\Om$.
We  assume that  $F,\,G\colon \mathbb{R}^n \to \mathbb{R}$ are $C^1$ integrands satisfying, for  $ p > 1$ and positive
constants  $\ell,\,L,\,\alpha,\,\beta >0$ and $\mu \geq 0$, the following growth and uniform strong $p$-convexity hypotheses:
$$
0\le F(\xi)\le L (\mu^2+| \xi|^2 )^{\frac{p}{2}}\,,   \leqno{\rm (F1)}
$$
$$
\int_{\Omega}F(\xi+\nabla \varphi)\, dx \ge \int_{\Omega}\Big( F(\xi) +\ell (\mu^2+| \xi|^2 +|\nabla\varphi|^2)^{\frac{p-2}{2}}|\nabla\varphi|^2\Big)\,dx\,,\leqno{\rm (F2)}
$$
and
$$
0\le G(\xi)\le \beta L (\mu^2+| \xi|^2 )^{\frac{p}{2}}\,,  \leqno{\rm (G1)}
$$
$$
\int_{\Omega} G(\xi+\nabla \varphi)\, dx \ge \int_{\Omega} \Big( G(\xi) +\alpha \ell (\mu^2+| \xi|^2 +|\nabla\varphi|^2)^{\frac{p-2}{2}}|\nabla\varphi|^2\Big)\,dx \leqno{\rm (G2)}
$$
for every $\xi\in\mathbb{R}^n$ and $\varphi\in C_0^1(\Omega)$.

  We are interested in  the following constrained problem
$$
\min\left\{{\mathcal I}(u,E):\,u=u_0\,\,\text{on}\,\,\partial\Om,\,|E|=d\right\}\,,\eqno (P)
$$
where $u_0\in W^{1,p}(\Om)$ and $0<d<|\Om|$ are  prescribed. Note that the strong  convexity of $F$ and $G$, expressed by (F2) and (G2), ensures the existence of solutions of the problem (P).
\par
\noindent Energies with surface terms competing with a volume term  appear in a plethora of phenomena in materials science such as  models for optimal design \cite{AB}, phase transitions \cite{Gur}, liquid crystals \cite{Lar}, epitaxy \cite{FFLMo} (see also \cite{FFLM}).

 Our first regularity result is the following:

\begin{Thm}\label{duedue}
Let $F$ and $G$ satisfy assumptions (F1)-(F2) and (G1)-(G2), respectively. Assume, in addition, that $F$ is $p$-homogeneous, i.e., $F(t\xi)=t^pF(\xi)$, for all $t\ge 0$.
If $(u,E)$ is a minimizer of problem (P), then $u\in C^{0,\frac{1}{p'}}_{\mathrm{loc}}(\Om)$, where $p'$ denotes the H\"older's conjugate exponent of $p$, i.e., $p'=\frac{p}{p-1}$. Moreover, $\H^{n-1}((\partial E\setminus \partial^*E)\cap\Omega)=0$.
\end{Thm}

\par
 Previous results in this direction have been obtained in \cite{AB} and \cite{Lin}.  Precisely, Ambrosio and Buttazzo  (\cite{AB}) and Lin (\cite{Lin}) considered problems of the form
\begin{equation}\label{pbab}
\int_\Om\bigl(\sigma_E(x)|\nabla u|^2\bigr)\,dx+P(E,\Om)
\end{equation}
with $u=0$ on $\partial\Om$ and  $ \sigma_E(x):=a\chi_E+b\chi_{\Omega\setminus E}$ for $a$ and $b$ positive constants. It was proven in \cite{AB}  that minimizers of \eqref{pbab} exist  and that if $(u,E)$ is a minimal configuration then $u$ is locally H\"older continuous in $\Om$ and, up to a set of $\mathcal{H}^{n-1}$ measure zero, there is no difference between the theoretic measure boundary of $E$ and its topological boundary. Recently, in \cite{EF}, it has been proven that there exists $\gamma=\gamma(n)$ such that,  for  a minimal configuration $(u,E)$ of \eqref{pbab} if $1<a/b<\gamma(n)$, then $u$ is locally H\"older continuous in $\Om$  and $\partial^*E$,  the reduced boundary of $E$, is a $C^{1,\alpha}$-hypersurface.
Moreover,  Lin (\cite{Lin}) showed that if $(u,E)$ is a minimizer of \eqref{pbab} among all configurations such that $u$ and $\partial E$ are prescribed on $\partial\Om$, then $u\in C^{0,1/2}(\Om)$ and $\partial^*E$,  the reduced boundary of $E$, is a $C^{1,\alpha}$-hypersurface away from a  singular set $\Sigma$ of $\H^{n-1}$ measure zero.
In \cite{KL}, Lin and Kohn establish a partial regularity result for the boundary of the minimal set of the  problem
\begin{equation}\label{intro1}
{\mathcal I}(u,E):=\int_\Om \left(F(x,u,\nabla u)+\chi_{_{E}}G(x,u,\nabla u)\right)\,dx+P(E,\Om)\,,
\end{equation}
 subject to the following constraints
 $$u=\Phi\,\,\,\mathrm{on}\,\,\,\partial\Omega\,\,\,\mathrm{and}\,\,\, |E|=d,$$
 requiring that $F$ and $G$ satisfy severe structure assumptions and have quadratic growth.
A  more detailed analysis of the minimal configurations of (P) was carried out in the two dimensional case by  Larsen in \cite{Lar}. However, also in this case only  partial regularity of $\partial^* E$ is obtained.
\par
 All minimum problems considered in the above mentioned  papers have  bulk energies of Dirichlet type with quadratic growth, i.e., of  the form $|\cdot|^2$.
Here, in Theorem \ref{duedue} we treat constrained problems,  we do not require any additional structure assumption on the bulk energies, and we  assume  $p$-growth (not necessarily  $p=2$) with respect to the gradient.

\noindent We point out that the H\"older exponent $\frac{1}{p'}$ in Theorem \ref{duedue} is critical, in the sense that the two terms in the energy functional \eqref{intro0} locally have the same dimension $n-1$ (under appropriate scalings). Actually,  we will show that $u\in C^{0,\frac{1}{p'}+\delta}_{\mathrm{loc}}(\Omega)$, for some $\delta>0$, under suitable conditions on the eigenvalues of $F$ and $G$, that, together with   a result in \cite{Tam} (see Theorem \ref{tre} in Section 2), allows us to conclude that $\partial^* E$ is a $C^{1,\widetilde\delta}$ hypersurface, for some $0<\widetilde\delta<1$. More precisely, we have

\begin{Thm}\label{due}
Let $F$ and $G$ satisfy assumptions (F1)-(F2) and (G1)-(G2), respectively. There exist $\gamma=\gamma(n,p,\frac{\ell}{L})<1$ and $\widetilde\sigma=\widetilde \sigma(n,p)$ such that if
\begin{equation}\label{due7prima}
\left(\frac{\beta}{\alpha+1}\right)\left(\frac{\beta+1}{\alpha+1}\right)^{\widetilde\sigma}\le \gamma
\end{equation}
and if $(u,E)$ is a minimizer of problem (P), then $u\in C^{0,\frac{1}{p'}+\delta}_{\mathrm{loc}}(\Om)$ for some positive $\delta$ depending on $n,p,\alpha,\beta$. Moreover $\partial^*E$ is a $C^{1,\widetilde\delta}$-hypersurface in $\Omega$, for some $\widetilde\delta<\frac{1}{2}$ depending on $n,p,\alpha,\beta$, and $\H^s((\partial E\setminus \partial^*E)\cap\Omega)=0$ for all $s>n-8$.
\end{Thm}

\noindent Consider the prototype integrands
$$F(\xi):= L(\mu^2+|\xi|^2)^{\frac{p}{2}}\qquad\qquad\mathrm{and}\qquad\qquad G(\xi):=\beta L (\mu^2+|\xi|^2)^{\frac{p}{2}}.$$
In this case the parameter $\alpha$  in assumption (G2) coincides with $\beta$ and condition  \eqref{due7prima} reduces to
$$\beta\le \frac{\gamma}{1-\gamma},$$
with $\gamma=\gamma(n,p)<1$.

\noindent The functional  \eqref{pbab} is a  particular case of \eqref{intro0}, setting
 $$F(\xi): = b|\xi|^2\qquad\qquad\mathrm{and}\qquad\qquad G(\xi):=(a-b) |\xi|^2,\qquad \quad a>b.$$
In this case, the parameters $\alpha,\beta$ in (G2) and (G1) are given by
  $$\beta=\alpha:=\frac{a-b}{b},$$
and  condition  \eqref{due7prima} becomes   $$1<\frac{a}{b}\le \frac{1}{1-\gamma}.$$
So, Theorem \ref{due} gives back  Theorem 2 in \cite{EF} as a particular case.

\noindent Further, without imposing any condition on the eigenvalues of the integrands, we are still able to obtain the following partial regularity result:
\begin{Thm}\label{duetre}
Assume that (F1)-(F2) and (G1)-(G2) hold and let
 $(u,E)$ be a minimizer of problem (P). Then there exists an open set $\Omega_0\subset \Omega$ with full measure such that $u\in C^{0,\eta}(\Om_0)$, for every positive  $\eta<1$. In addition, $\partial^*E\cap \Om_0$ is a $C^{1,\widetilde\eta}$-hypersurface in $\Omega_0$, for every $0<\widetilde\eta<\frac{1}{2}$, and $\H^{s}((\partial E\setminus \partial^*E)\cap\Om_0)=0$  for all $s>n-8$.
\end{Thm}

 In the study   of regularity properties,  the constraint $|E|=d$ introduces extra difficulties, since one can work only with variations which keep the volume constant. The next theorem allows us to circumvent this extra difficulties, ensuring that every minimizer of the constrained problem (P) is also a minimizer of a suitable unconstrained energy functional with a volume penalization.

\begin{Thm}\label{uno}
There exists $\lambda_0>0$ such that if  $(u,E)$ is  a minimizer of the functional
\begin{equation}\label{unozero}
{\mathcal I}_\lambda(v,A):=\int_\Om F(\nabla v)+\chi_{_{A}}G(\nabla v)\,dx+P(A,\Om)+\lambda\big||A|-d\big|
\end{equation}
for some $\lambda\ge \lambda_0$ and
among all configurations $(v,A)$ such that $v=u_0$ on $\partial\Om$, then $|E|=d$ and $(u,E)$ is a minimizer of problem (P). Conversely, if $(u,E)$ is a minimizer of the problem (P), then it is a minimizer of \eqref{unozero}, for all $\lambda\geq\lambda_0$.
\end{Thm}
\par
 Theorem \ref{uno} is a straightforward  modification of  a result due to Esposito and Fusco (see \cite[Theorem1]{EF}). Since several modifications are needed, we present its proof in Section 1 for the reader's convenience. Similar arguments have been used in Fonseca, Fusco, Leoni and Millot (\cite{FFLM}) (see also Alt and Caffarelli \cite{AC}).

\noindent From the point of view of  regularity, the extra term $\lambda\big||A|-d\big|$ is a higher order, negligible perturbation, in the sense that if $x_0\in\partial^*E\cap\partial\Om$ then $|E\cap B_\ro(x_0)|$ decays as $\ro^n$ as $\ro\to0^+$ while the leading term
$\int_{B_\ro(x_0)}\!\left(F(\nabla u)+\chi_{_{E}}G(\nabla u)\right)\,dx+P(E,B_\ro(x_0))$ decays as $\ro^{n-1}$.

 The proof of Theorem \ref{duedue} is based on a decay estimate for the gradient of the minimizer $u$, obtained by  blowing-up the minimizer $u$ in small balls. We  establish that the minimizers of the rescaled  problems  converge to a H\"older continuous  function $v$, and  we show that $u$ and
$v$ are "close enough"  (with respect to the norm in the Sobolev space $W^{1,p}$) in order to ensure that $u$ inherits the
regularity estimates of $v$.

\noindent Theorems \ref{due}  and  \ref{duetre} are obtained by a comparison argument between the minimizer of (P) and the minimizer of a suitable convex scalar functional with $p$-growth, for which regularity results are well known. Also here,
we show that the  two minimizers  are "close" enough  to share the same good
regularity properties.
We remark that in this comparison argument we need that $u$ is a real valued function. In  fact, in the vectorial setting (see \cite{sy}) minimizers of regular variational functionals may have singularities and only partial regularity results are known (see for example \cite{af1,cp,fh}).

 This paper is organized as follows: In  Section 2 we fix the notation and collect standard preliminary results. The proof of Theorem \ref{uno} is given in Section 3, since the result is needed in the proofs of the other theorems. The proofs of Theorems \ref{duedue}, \ref{due} and \ref{duetre} are presented in Sections 4, 5 and 6, respectively.

\bigskip

\section{Notations and Preliminary Results}

\bigskip

In this paper we follow  usual convention and denote by $c$ a
general constant that may vary from expression to expression, even within the
same line of estimates.
Relevant dependencies on parameters and special constants will be suitably emphasized using
parentheses or subscripts. The norm we use
in $\mathbb{R}^n$  is the standard Euclidean norm, and it will be denoted
by $| \cdot |$. In particular, for vectors $\xi$, $\eta \in \mathbb{R}^n$ we write
$\langle \xi , \eta \rangle $ for the  inner product
of $\xi$ and $\eta$, and $| \xi | := \langle \xi , \xi \rangle^{\frac{1}{2}}$ is the
corresponding Euclidean norm.
When $a , b \in \mathbb{R}^n$ we write $a \otimes b $ for the tensor product defined
as the matrix that has the element $a_{r}b_{s}$ in its r-th row and s-th column. Observe that
$(a \otimes b)x = (b \cdot x)a$ for $x \in \mathbb{R}^n$, and $| a \otimes b | = |a| |b|$.
When $F \colon \mathbb{R}^n \to \R$ is $C^1$, we write
$$
D_\xi F(\xi )[\eta ] := \frac{\rm d}{{\rm d}t}\Big|_{t=0} F(\xi +t\eta )
$$
if $\xi$, $\eta \in \mathbb{R}^n$.
It is convenient to express the convexity and growth conditions of the
integrands in terms of an auxiliary function defined for all $\xi \in \mathbb{R}^n$ as
\begin{equation}\label{aux2}
V( \xi ) = V_{p,\mu}(\xi ) := \Bigl( \mu^{2}+| \xi |^{2} \Bigr)^{\frac{p-2}{4}}\xi ,
\end{equation}
where $\mu \geq 0$ and $p \geq 1$.
We recall the following
lemmas.

\begin{Lem} \label{V}
Let $1<p<\infty$ and $0\leq \mu \leq 1$. There exists a constant $c=c(n,N,p)>0$
such that
$$
c^{-1}\Bigl( \mu^{2}+| \xi |^2+| \eta |^2 \Bigr)^{\frac{p-2}{2}}\leq
\frac{|V_{p,\mu}(\xi )-V_{p,\mu}(\eta )|^2}{|\xi -\eta |^2} \leq
c\Bigl( \mu^{2}+|\xi |^2+|\eta |^2 \Bigr)^{\frac{p-2}{2}}
$$
for all $\xi$, $\eta \in \mathbb{R}^n$.
\end{Lem}
\noindent For  the proof we refer to \cite[Lemma 8.3]{gi}.
The next lemma  can be found in \cite[Lemma 2.1]{GM} and \cite[Lemma 2.1]{AF} for $p\ge 2$ and
$1<p<2$, respectively.
\begin{Lem}\label{V1}
 For $1 < p < \infty$ and all $\xi$, $\eta \in \mathbb{R}^n$ one has
$$\frac{1}{c}\le \frac{\int^1
_0 (\mu^2 + |\xi + t \eta)|^2)^{\frac{p-2}{2}}
dt}{(\mu^2 + |\xi|^2 + |\eta|^2)
^{\frac{p-2}{2}}}\le c,
$$
where c depends only on p.
\end{Lem}
\noindent It is well--known that for convex $C^1$ integrands, the assumptions (F1) and (G1) yield the upper bounds
\begin{equation}\label{(H4)}
|D_\xi F(\xi)|\le c_1(\mu^2+|\xi|^2)^{\frac{p-1}{2}}
\qquad\qquad
|D_\xi G(\xi)|\le  c_2(\mu^2+|\xi|^2)^{\frac{p-1}{2}}
\end{equation}
for all $\xi \in \R^n$, where we can use $c_1: =2^{p}L$ and $c_2: =2^{p}\beta L$ (see \cite{gi}).

\noindent Also, if $F$ and $G$   satisfy (F2) and (G2), respectively, then the following strong $p$-monotonicity conditions hold:
\begin{eqnarray}\label{deriv}
& &\langle D_\xi F(\xi)-D_\xi F(\eta),\xi-\eta\rangle\ge c(p)\ell |V(\xi)-V(\eta)|^2\cr\cr & &\langle D_\xi G(\xi)-D_\xi G(\eta),\xi-\eta\rangle\ge c(p)\alpha\ell |V(\xi)-V(\eta)|^2
\end{eqnarray}
for all $\xi$, $\eta \in \R^n $ and some $c(p)>0$. In fact, (F2) and (G2) are equivalent to the convexity of the functions
$$\xi\to \widetilde F(\xi):=F(\xi)-\ell(\mu^2+| \xi|^2 )^{\frac{p}{2}}$$
and
$$\xi\to \widetilde G(\xi):=G(\xi)-\alpha\ell(\mu^2+| \xi|^2 )^{\frac{p}{2}},$$
respectively (see for example \cite{gi}, p.164). Hence, the convexity of $\widetilde F$ implies
$$F(\xi)-\ell(\mu^2+| \xi|^2 )^{\frac{p}{2}}\ge F(\eta)-\ell(\mu^2+| \eta|^2 )^{\frac{p}{2}}+\langle D_\xi F(\eta),\xi-\eta\rangle-\ell p\langle (\mu^2+| \eta|^2 )^{\frac{p}{2}-1}\eta,\xi-\eta\rangle$$
and
$$F(\eta)-\ell(\mu^2+| \eta|^2 )^{\frac{p}{2}}\ge F(\xi)-\ell(\mu^2+| \xi|^2 )^{\frac{p}{2}}+\langle D_\xi F(\xi),\eta-\xi\rangle-\ell p\langle (\mu^2+| \xi|^2 )^{\frac{p}{2}-1}\xi,\eta-\xi\rangle.$$
Summing previous inequalities and using Lemmas \ref{V1} and \ref{V}, we obtain
\begin{eqnarray*}\langle D_\xi F(\xi)-D_\xi F(\eta),\xi-\eta\rangle &\ge& \ell p\langle (\mu^2+| \xi|^2 )^{\frac{p}{2}-1}\xi,\xi-\eta\rangle-\ell p\langle (\mu^2+| \eta|^2 )^{\frac{p}{2}-1}\eta,\xi-\eta\rangle\cr\cr
&\ge& \ell p\langle (\mu^2+| \xi|^2 )^{\frac{p}{2}-1}\xi-(\mu^2+| \eta|^2 )^{\frac{p}{2}-1}\eta,\xi-\eta\rangle\cr\cr
&\ge &c(p)\ell \int_0^1(\mu^2+| \xi+t(\eta-\xi)|^2 )^{\frac{p}{2}-1}\,dt|\xi-\eta|^2\cr\cr
&\ge&c(p)\ell |V(\xi)-V(\eta)|^2,
\end{eqnarray*}
i.e, the first inequality in \eqref{deriv}. The second inequality in \eqref{deriv} can be derived arguing  similarly.

\noindent Further, if $F$ and $G$ are $C^2$, then
(F2) and (G2) are equivalent to the following standard strong $p$--ellipticity conditions
$$
\langle D^2F(\xi)\eta,\eta \rangle\ge c_3(\mu^2+|\xi|^2)^{\frac{p-2}{2}}|\eta |^{2},
\qquad\qquad
\langle D^2G(\xi)\eta,\eta \rangle\ge c_4(\mu^2+|\xi|^2)^{\frac{p-2}{2}}|\eta |^{2}
$$
for all $\xi$, $\eta \in \R^n $, where $c_i$ are positive constants of form $c_3=c(p)\ell$ and $c_4=c(p)\alpha\ell$, respectively.

 The next lemma establishes that the uniform strong $p$--convexity assumptions  (F2) and (G2) yield  growth conditions from below for the
functions $F$ and $G$.

\begin{Lem}\label{grobelow}
Suppose that $H:\mathbb{R}^n\to [0,+\infty)$ is a $C^1$ function such that
\begin{equation}\label{pgro}0\le H(\xi)\le \widetilde L (\mu^2+| \xi|^2 )^{\frac{p}{2}} \end{equation}
for all $\xi\in\mathbb{R}^n$, where $p>1$, $0\le \mu\le 1$, $ \widetilde L>0$. Assume, in addition, that
\begin{equation}\label{pconv}\int_Q H(\xi+\nabla \varphi)\, dx \ge \int_Q H(\xi) +\widetilde \ell (\mu^2+| \xi|^2 +|\nabla\varphi|^2)^{\frac{p-2}{2}}|\nabla\varphi|^2\,dx\end{equation}
for all $\xi\in\mathbb{R}^n$, $\varphi\in C_0^1(Q)$, $Q\subset \mathbb{R}^n$ and for some positive constant $\widetilde \ell$.
Then there exists a positive constant $c(p,\widetilde L,\widetilde \ell, \mu)$ such that
\begin{equation}
H(\xi)\ge \frac{\widetilde\ell}{2}(\mu^2+| \xi|^2 )^{\frac{p}{2}}-c(p,\widetilde L,\widetilde \ell, \mu)\quad {for\,\,all\,\,}\xi\in \mathbb{R}^n.
\end{equation}
\end{Lem}
\begin{proof}
We use again the fact that assumption \eqref{pconv} is equivalent to the convexity of the function
$$\xi\to K(\xi):=H(\xi)-\widetilde \ell(\mu^2+| \xi|^2 )^{\frac{p}{2}}.$$
Hence
\begin{equation*}
K(\xi)\ge K(0)+\langle D_\xi K(0),\xi\rangle,
\end{equation*}
or, equivalently,
\begin{equation}\label{gro1}
H(\xi)\ge \widetilde \ell(\mu^2+| \xi|^2 )^{\frac{p}{2}}+H(0)-\widetilde\ell\mu^p+\langle D_\xi H(0),\xi\rangle
\end{equation}
for all $\xi\in \R^n $. By \eqref{pgro} and  \eqref{(H4)}, we have that
$$H(0)\ge 0\qquad\mathrm{and}\qquad |D_\xi H(0)|\le 2^p\widetilde L \mu^{p-1},$$
and by Young's inequality
\begin{eqnarray}\label{gro2}
\Big|\langle D_\xi H(0),\xi\rangle \Big|&=& \Big|\langle \big(\,\widetilde \ell\,\big)^{-\frac{1}{p}}D_\xi H(0),\big(\,\widetilde \ell\,\big)^{\frac{1}{p}}\xi\rangle \Big|\le c(\varepsilon)\big(\,\widetilde \ell\,\big)^{-\frac{1}{p-1}}|D_\xi H(0)|^{\frac{p}{p-1}}+\varepsilon \widetilde \ell|\xi|^p\cr\cr
&\le& c(\varepsilon)2^{\frac{p^2}{p-1}}\big(\,\widetilde \ell\,\big)^{-\frac{1}{p-1}}\widetilde L^{\frac{p}{p-1}} \mu^{p}+\varepsilon \widetilde \ell(\mu^2+|\xi|^2)^\frac{p}{2}.\end{eqnarray}
Inserting \eqref{gro2} in \eqref{gro1}, we get
$$H(\xi)\ge \widetilde \ell(\mu^2+| \xi|^2 )^{\frac{p}{2}}-c(\varepsilon)2^{\frac{p^2}{p-1}}\big(\,\widetilde \ell\,\big)^{-\frac{1}{p-1}}\widetilde L^{\frac{p}{p-1}} \mu^{p}-\varepsilon \widetilde \ell(\mu^2+|\xi|^2)^\frac{p}{2}-\widetilde\ell\mu^p$$
and, choosing $\varepsilon=\frac{1}{2}$, we conclude that
$$H(\xi)\ge \frac{\widetilde \ell}{2}(\mu^2+| \xi|^2 )^{\frac{p}{2}}-c_p\left(\frac{\widetilde L^p}{\widetilde \ell}\right)^{\frac{1}{p-1}} \mu^{p}-\widetilde\ell\mu^p.$$
\end{proof}
\noindent As already mentioned in the Introduction, we will compare the minimizer $u$ of the problem (P) with the minimizer of a suitable regular convex variational integral. In order to take advantage of the comparison argument, we will  need the following regularity result (see \cite[Theorem 2.2]{ff})
\begin{Thm}\label{fofu}
Let $H:\mathbb{R}^n\to [0,+\infty)$ be a continuous function such that
$$0\le H(\xi)\le \widetilde L (\mu^2+| \xi|^2 )^{\frac{p}{2}} $$
for all $\xi\in\mathbb{R}^n$, where $p>1$, $0\le \mu\le 1$, $ \widetilde L>0$.
Suppose, in addition, that
$$\int_Q H(\xi+\nabla \varphi)\, dx \ge \int_Q H(\xi) +\widetilde \ell (\mu^2+| \xi|^2 +|\nabla\varphi|^2)^{\frac{p-2}{2}}|\nabla\varphi|^2\,dx$$
for all $\xi\in\mathbb{R}^n$, $\varphi\in C_0^1(Q)$, $Q\subset \mathbb{R}^n$ and for some positive constant $\widetilde \ell$. If $v\in W^{1,p}(\Omega)$ is a local minimizer of the functional $$\mathcal{H}(w,\Omega):=\int_\Omega H(\nabla w)\, dx,$$
i.e.,
$$\mathcal{H}(v,B_r(x_0))=\min\Big\{\mathcal{H}(w,B_r(x_0))\,\,:\,\, w\in v+W^{1,p}_0(B_r(x_0))\Big\}\qquad {for\,\,all}\,\, B_r(x_0)\subset\Omega\,,$$
then $v$ is locally Lipschitz in $\Omega$, and
\begin{equation}\label{stimafofu}\mathrm{ess}\!\!\sup_{B_{\frac{R}{2}}(x_0)}(\mu^2+| \nabla v|^2 )^{\frac{p}{2}}\le c (n,\widetilde L,\widetilde\ell,p) \medint_{B_R(x_0)}(\mu^2+| \nabla v|^2 )^{\frac{p}{2}}\,dx
\end{equation}
for every $B_{\frac{R}{2}}(x_0)\subset B_R(x_0)\subset\Omega$.
\end{Thm}
\noindent In what follows, we will need  a more explicit dependence on the eigenvalues of $H$ of the constant in \eqref{stimafofu}. Actually, a careful inspection of the proof of Theorem 2.2 in \cite{ff} reveals that the constant in estimate \eqref{stimafofu} is of the type
\begin{equation}\label{cost}c (n,\widetilde L,\widetilde\ell,p)=c \left(\frac{\widetilde L}{\widetilde \ell}\right)^{\frac{2n}{p}} \end{equation}
where $c=c(n,p) \ge 1$.

\noindent The following is a technical iteration lemma (see  \cite[Lemma 7.3]{gi})
\begin{Lem}\label{iter}
Let  $\varphi$ be a nonnegative, nondecreasing function and assume that there exist  $\vartheta\in(0,1)$, $\bar R>0$, and  $0<\beta<\gamma$ such that
$$\varphi(\vartheta r)\leq \vartheta^\gamma\varphi( r)+br^{\beta}$$   for all $0<r\leq \bar R$. Then  we have
  $$\varphi(\rho)\leq C\left\{\left(\frac{\rho}{r}\right)^{\beta}\varphi(r)+b\rho^{\beta}\right\},$$
  for every $0<\rho<r\leq \bar R$, with $C=C(\vartheta,\beta,\gamma)$.
\end{Lem}
\noindent The next result relates the decay estimate for the gradient of a Sobolev function with its H\"older regularity properties  (see \cite[Theorem 1.1, p. 64]{Gia}, \cite[Theorem 7.19]{GTr})

\begin{Thm}[Morrey's Lemma]\label{Morrey} Let $u\in W^{1,1}(\Omega)$  and suppose that there exist positive constants $K,\, 0<\alpha\le 1$ such that   $$\int_{B_r(x)}|\nabla u|\,dx\le K r^{n-1+\alpha},$$
for all balls $B(x,r)\subset \Omega$, $x\in \Omega$,  $r>0$.
Then $u\in C^{0,\alpha}(\Omega)$.
\end{Thm}

\noindent Given a Borel set $E$ in $\mathbb{R}^n$,  $P(E,\Omega)$ denotes the perimeter of $E$ in $\Omega$,  defined as
$$P(E,\Omega)=\sup\left\{\int_E \mathrm{div}\phi\,dx:\,\,\, \phi\in C^1_0(\Omega;\mathbb{R}^n), \,\, |\phi|\le 1\right\}.$$
It is known that, for a set of finite perimeter $E$, one has
$$P(E,\Omega)=\mathcal{H}^{n-1}(\partial^* E)$$
where
$$\partial^* E=\left\{ x\in \Omega:\,\, \limsup_{\rho\to 0^+}\frac{P(E,B_\rho(x))}{\rho^{n-1}}>0\right\}$$
is the reduced boundary of $E$ (for more details we refer to \cite{AFP}).

 \noindent Given a set $E\subset\Om$ of finite perimeter in $\Om$, for every ball $B_r(x)\Subset\Om$ we measure how far $E$ is from being an area minimizer in the ball by setting
$$
\psi(E,B_r(x)):=P(E,B_r(x))-\min\left\{P(A,B_r(x)):\,A\Delta E\Subset B_r(x),\,\,\chi_A \in BV(\mathbb{R}^n)\right\}\,.
$$
The following regularity result, due to Tamanini (see \cite{Tam}), asserts that if the excess $\psi(E,B_r(x))$ decays fast enough when $r\to 0$, then $E$ has essentially the same regularity properties of an area minimizing set.
\begin{Thm}\label{tre}
Let $\Om$ be an open subset of $\R^n$ and let $E$ be a set of finite perimeter satisfying, for some $\delta\in (0,\frac{1}{2})$,
$$
\psi(E,B_r(x))\leq cr^{n-1+\delta}
$$
for every $x\in\Om$ and every $r\in(0,r_0)$, with $c=c(x),r_0=r_0(x)$  local positive constants. Then $\partial^*E$ is a $C^{1,\delta}$-hypersurface in $\Om$ and $\H^s\left((\partial E\setminus \partial^*E)\cap\Om\right))=0$ for all $s>n-8$.
\end{Thm}

\section{Proof of Theorem \ref{uno}}

\bigskip

This section is devoted to the proof of Theorem \ref{uno}, which follows closely  that of Theorem 1 in \cite{EF}. Since several  modifications are needed, we present it here for the  convenience of the reader.

\bigskip

\begin{proof}[Proof of Theorem \ref{uno}]

 \noindent {\bf Step 1.}\quad We prove the first part of Theorem~\ref{uno} arguing by contradiction. Assume that there exist a sequence $\{\lambda_h\}_{h\in\N}$ such that $\lambda_h\to\infty$ as $h\to\infty$, and a sequence of configurations $\{(u_h,E_h)\}$ minimizing ${\mathcal I}_{\lambda_h}$ such that $u_h=u_0$ on $\partial \Omega$ and $|E_h|\not=d$ for all $h$. Notice that
\begin{equation}\label{unouno}
{\mathcal I}_{\lambda_h}(u_h,E_h)\leq{\mathcal I}(u_0,E_0)=:\Theta\,,
\end{equation}
where $E_0\subset\Om$ is a fixed set of finite perimeter such that $|E_0|=d$. Assume that $|E_h|<d$ for a (not relabeled) subsequence  (if $|E_h|>d$ the proof is similar). We claim that, for $h$ sufficiently large, there exists a configuration $({\tilde u}_h,{\widetilde E}_h)$ such that ${\mathcal I}_{\lambda_h}({\tilde u}_h,{\widetilde E}_h)<{\mathcal I }_{\lambda_h}(u_h,E_h)$, thus proving that $|E_h|=d$ for all $h$ sufficiently large, say $\lambda\ge \lambda_0$.

 By our assumptions on $F$ and $G$, it follows that the sequence $\{u_h\}$ is bounded in  $W^{1,p}(\Om)$, the perimeters of the sets $E_h$ are bounded, and $|E_h|\to d$. Therefore, without  loss of generality, we may assume, possibly extracting  a subsequence (not relabeled), that there exists a configuration $(u,E)$ such that $u_h\to u$ weakly in
$W^{1,p}(\Om)$, $\chi_{E_h}\to\chi_E$ a.e. in $\Om$, and  $E$ is a set of finite perimeter in $\Om$ with $|E|=d$.

\medskip

\noindent {\bf Step 2.}\,\,{\it Construction of  $({\tilde u}_h,{\widetilde E}_h)$.}\quad
Fix a point $x\in\partial^*E\cap\Om$ (such a point exists since $E$ has finite perimeter in $\Om$, $0<|E|<|\Om|$, and $\Omega$ is connected). By De Giorgi's structure theorem for sets of finite perimeter (see \cite[Theorem~3.59]{AFP}), the sets $E_r=(E-x)/r$ converge locally in measure to the half space $H=\{z\cdot\nu_E(x)>0\}$, i.e., $\chi_{E_r}\to\chi_H$ in $L^1_{\rm loc}(\R^n)$, where $\nu_E(x)$ is the generalized inner normal to $E$ at $x$ (see \cite[Definition~3.54]{AFP}). Let $y\in B_1(0)\setminus H$ be the point $y:=-\nu_E(x)/2$. Given $\e>0$ and small (to be chosen at the end of the proof), since $\chi_{E_r}\to\chi_H$ in $L^1(B_1(0))$ there exists $r>0$ such that
$$
B_{2r}(x)\subset \Omega, \qquad |E_r\cap B_{1/2}(y)|<\e,\qquad |E_r\cap B_1(y)|\geq|E_r\cap B_{1/2}(0)|>\frac{\omega_n}{2^{n+2}}\,,
$$
where $\omega_n$ denotes the measure of the unit ball of $\R^n$. Therefore, setting  $x_r:=x+ry\in\Om$, we have
$$
B_{r}(x_r)\subset \Omega, \qquad |E\cap B_{r/2}(x_r)|<\e r^n,\qquad|E\cap B_r(x_r)|>\frac{\omega_nr^n}{2^{n+2}}\,.
$$
Assume, without loss of generality, that $x_r=0$, and in the sequel denote the open  ball centered at the origin and with radius $r>0$ by $B_r$. From the convergence of $\{E_h\}$ to $E$ we have that, for all $h$ sufficiently large,
\begin{equation}\label{unodue}
|E_h\cap B_{r/2}|<\e r^n,\qquad|E_h\cap B_r|>\frac{\omega_nr^n}{2^{n+2}}\,.
\end{equation}
Define the  bi-Lipschitz map  $\phi: B_r\to B_r$ by
\begin{equation}\label{unotre}
\Phi(x):=
\begin{cases}
\bigl(1-\sigma(2^n-1)\bigr)x & \text{if}\,\,\,|x|<\displaystyle\frac{r}{2},\\
x+\sigma\Bigl(\displaystyle1-\frac{r^n}{|x|^n}\Bigr)x & \text{if}\,\,\,\displaystyle\frac{r}{2}\leq|x|<r,\\
x & \text{if}\,\,\,|x|\geq r\,,
\end{cases}
\end{equation}
for some fixed $0<\sigma<1/2^n$, to be determined later, such that, setting
$$
{\widetilde E}_h:=\Phi(E_h),\qquad{\tilde u}_h:=u_h\circ\Phi^{-1}\,,
$$
we have $|{\widetilde E}_h|<d$. We obtain
\begin{eqnarray}\label{unoquattro}
& &{\mathcal I}_{\lambda_h}(u_h,E_h)-{\mathcal I}_{\lambda_h}({\tilde u}_h,{\widetilde E}_h)\nonumber \\
&=&\displaystyle\biggl[\int_{B_r}\!\Big(F(\nabla u_h)+\chi_{_{E_h}}G(\nabla u_h)\Big)\,dx-\int_{B_r}\!\Big(F(\nabla{\tilde u}_h)+\chi_{_{\widetilde E_h}}G(\nabla{\tilde u}_h)\Big)
\,dy\biggr] \\
&& \quad+\bigl(P(E_h,{\overline B}_r)-P({\widetilde E}_h,{\overline B}_r)\bigr)+\lambda_h\bigl(|{\widetilde E}_h|-|E_h|\bigr) \nonumber \\
\!\!\!&=:&\!\!\! I_{1,h}+I_{2,h}+I_{3,h}\,. \nonumber
\end{eqnarray}
{\bf Step 3.}\,\,{\it Estimate of  $I_{1,h}$.}
 We start by evaluating the gradient and the Jacobian determinant of $\Phi$ in the annulus $B_r\setminus B_{r/2}$. If $r/2<|x|<r$, then we have
$$
\frac{\partial\Phi_i}{\partial x_j}(x)=\Bigl(1+\sigma-\frac{\sigma r^n}{|x|^n}\Bigr)\delta_{ij}+n\sigma r^n\frac{x_ix_j}{|x|^{n+2}}\qquad\qquad \text{for all}\,\,\,i,j=1,\dots n
$$
and thus, if $\eta\in\R^n$,
$$
(\nabla\Phi\circ\eta)\cdot\eta=\Bigl(1+\sigma-\frac{\sigma r^n}{|x|^n}\Bigr)|\eta|^2+n\sigma r^n\frac{(x\cdot\eta)^2}{|x|^{n+2}}
$$
from which it follows that
$$
|\nabla\Phi\circ\eta|\geq\Bigl(1+\sigma-\frac{\sigma r^n}{|x|^n}\Bigr)|\eta|\,.
$$
From this inequality we easily deduce an estimate on the norm of $\nabla\Phi^{-1}$, precisely,
\begin{eqnarray}\label{unocinque}
\bigl\|\nabla\Phi^{-1}\bigl(\Phi(x)\bigr)\bigr\| \!\!\!&=&\!\!\! \max_{|\eta|=1}\biggl|\nabla\Phi^{-1}\circ\Bigl(\frac{\nabla\Phi\circ\eta}{|\nabla\Phi\circ\eta|}\Bigr)\biggr|=\max_{|\eta|=1}\frac{1}{|\nabla\Phi\circ\eta|} \\
\!\!\!&\leq&\!\!\! \Bigl(1+\sigma-\frac{\sigma r^n}{|x|^n}\Bigr)^{-1}\leq\bigl(1-(2^n-1)\sigma\bigr)^{-1}\qquad\text{for all}\,\,\,x\in B_r\setminus B_{r/2}\,. \nonumber
\end{eqnarray}
Concerning the Jacobian, we write, for $x\in B_r\setminus B_{r/2}$,
\begin{equation}\label{unocinque.1}
\Phi(x)=\varphi(|x|)\frac{x}{|x|}\,,
\end{equation}
where
$$
\varphi(t)=t\Bigl(1+\sigma-\frac{\sigma r^n}{t^n}\Bigr),\qquad\qquad \text{for all}\,\,\,t\in[r/2,r]\,.
$$
Let  $I$ denote the identity map in $\R^n$. Recalling that if $A=I+a\otimes b$ for some vectors $a,\,b\in\R^n$,  then ${\rm det}A=1+a\cdot b$, a straightforward calculation gives for all $x\in B_r\setminus B_{r/2}$
\begin{equation}\label{unosei}
J\Phi(x)=\varphi^\prime(|x|)\Bigl(\frac{\varphi(|x|)}{|x|}\Bigr)^{n-1}=\Bigl(1+\sigma+\frac{(n-1)\sigma r^n}{|x|^n}\Bigr)\Bigl(1+\sigma-\frac{\sigma r^n}{|x|^n}\Bigr)^{n-1}\,.
\end{equation}
We have
\begin{eqnarray}\label{unosette}
\Bigl(1+\sigma-\frac{\sigma r^n}{|x|^n}\Bigr)^{n-1}&=&\Bigl(1+\sigma\Bigr)^{n-1}\Bigl(1-\frac{\frac{\sigma r^n}{|x|^n}}{1+\sigma}\Bigr)^{n-1}\ge \Bigl(1+\sigma\Bigr)^{n-1}\Bigl(1-(n-1)\frac{\frac{\sigma r^n}{|x|^n}}{1+\sigma}\Bigr)\cr\cr
&=&\Bigl(1+\sigma\Bigr)^{n-2}\Bigl(1+\sigma-(n-1)\frac{\sigma r^n}{|x|^n}\Bigr)\ge 1+\sigma-(n-1)\frac{\sigma r^n}{|x|^n}.
\end{eqnarray}
Since $x\in B_r\setminus B_{\frac{r}{2}}$, by \eqref{unosei} and \eqref{unosette} we have
\begin{eqnarray}\label{unootto}
J\Phi(x)&\ge& \left(1+\sigma+(n-1)\frac{\sigma r^n}{|x|^n}\right)\left(1+\sigma-(n-1)\frac{\sigma r^n}{|x|^n}\right)=
(1+\sigma)^2-(n-1)^2\frac{\sigma^2 r^{2n}}{|x|^{2n}}\cr\cr
&\ge&(1+\sigma)^2-4^n(n-1)^2\sigma^2=1+2\sigma-\Big(4^n(n-1)^2-1\Big)\sigma^2>1+\sigma
\end{eqnarray}
provided that we chose
$$\sigma <\frac{1}{4^n(n-1)^2-1}.$$ On the other hand, from \eqref{unosei} we get also
\begin{equation}\label{unosette.1}
J\Phi(x)\leq1+2^nn\sigma\,.
\end{equation}
Let us now turn to the estimate of $I_{1,h}$. Performing the change of variable $y=\Phi(x)$ in the second integral defining $I_{1,h}$, and observing that $\chi_{{\widetilde E}_h}(\Phi(x))=\chi_{E_h}(x)$, we get
\begin{eqnarray}
I_{1,h}&=&\int_{B_r}\!\Bigl[F(\nabla u_h(x))-F\bigl(\nabla u_h(x)\circ\nabla\Phi^{-1}\bigl(\Phi(x)\bigl)\bigl)J\Phi(x)\cr\cr
&& \qquad+\chi_{E_h}(x)\bigl[G(\nabla u_h(x))-G\bigl(\nabla u_h(x)\circ\nabla\Phi^{-1}\bigl(\Phi(x)\bigl)\bigl)J\Phi(x)\Bigr]\,dx\cr\cr
&=:&A_{1,h}+A_{2,h}\,,
\end{eqnarray}
where $A_{1,h}$ stands for the above integral evaluated in $B_{r/2}$ and $A_{2,h}$ for the same integral evaluated in $B_r\setminus B_{r/2}$. Recalling the definition of $\Phi$ in \eqref{unotre} and the growth assumptions on $F,G$ in (F1) and (G1), respectively, we get
\begin{eqnarray*}
A_{1,h}\!\!\!&=&\!\!\!\int_{B_{r/2}}\!\bigl[F(\nabla u_h(x))-F\bigl(\nabla u_h(x)\circ\bigl(1-\sigma(2^n-1)\bigr)^{-1}I\bigr)\bigl(1-\sigma(2^n-1)\bigr)^n\bigr]\,dx \\
\!\!\!&+&\!\!\!\int_{B_{r/2}}\!\chi_{E_h}(x)\bigl[G(\nabla u_h(x))-G\bigl(\nabla u_h(x)\circ\bigl(1-\sigma(2^n-1)\bigr)^{-1}I\bigr)\bigl(1-\sigma(2^n-1)\bigr)^n\bigr]\,dx \\
\!\!\!&\ge&\!\!\! -\int_{B_{r/2}}\!F(\nabla u_h(x)\circ\bigl(1-\sigma(2^n-1)\bigr)^{-1}I\bigr)\bigl(1-\sigma(2^n-1)\bigr)^n\,dx\\
\!\!\!& &\!\!\! -\int_{B_{r/2}}\!\chi_{E_h}(x)G(\nabla u_h(x)\circ\bigl(1-\sigma(2^n-1)\bigr)^{-1}I\bigr)\bigl(1-\sigma(2^n-1)\bigr)^n\,dx\\
&\geq& - c(p,\beta,L) \int_{B_{r/2}}\!(1+\chi_{E_h}(x))\bigl|\nabla u_h(x)\circ\bigl(1-\sigma(2^n-1)\bigr)^{-1}I\bigr|^p\bigl(1-\sigma(2^n-1)\bigr)^n\,dx\\
\!\!\!& &\!\!\! - c(p,\beta,L)\mu^p \bigl(1-\sigma(2^n-1)\bigr)^n r^n\\
&\ge& - c(p,\beta,L)\int_{B_{r/2}}\!(1+\chi_{E_h}(x))\bigl|\nabla u_h(x)\bigr|^p\bigl(1-\sigma(2^n-1)\bigr)^{n-p}\,dx\\
\!\!\!& &\!\!\! - c(p,\beta,L)\mu^p \bigl(1-\sigma(2^n-1)\bigr)^n r^n\\
&=&-c \bigl(1-\sigma(2^n-1)\bigr)^{n-p}\int_{B_{r/2}}\!\bigl|\nabla u_h(x)\bigr|^p\, dx- c\mu^p \bigl(1-\sigma(2^n-1)\bigr)^n r^n\\
\!\!\!&\ge&\!\!\! -C(n,p,\beta,L,\sigma,\mu)(\Theta+r^n)
\end{eqnarray*}
where we used \eqref{unouno}.
Recalling \eqref{unocinque}, \eqref{unosette.1} and \eqref{unouno} we have
\begin{eqnarray*}
A_{2,h}\!\!\!&=&\!\!\!\int_{B_r\setminus B_{r/2}}\!\bigl[F(\nabla u_h(x))-F\Bigl(\nabla u_h(x)\circ\nabla\Phi^{-1}\bigl(\Phi(x)\bigr)\Bigr)J\Phi(x)\bigr]\,dx \\
&+&\!\!\!\int_{B_r\setminus B_{r/2}}\!\chi_{E_h}(x)\bigl[G(\nabla u_h(x))-G\Bigl(\nabla u_h(x)\circ\nabla\Phi^{-1}\bigl(\Phi(x)\bigr)\Bigr)J\Phi(x)\bigr]\,dx \\
\!\!\!&\geq&- c(p,\beta,L)\!\!\!\int_{B_r\setminus B_{r/2}}\!(1+\chi_{E_h}(x))|\nabla u_h(x)|^p\bigl(1-(2^n-1)\sigma\bigr)^{-p}\bigl(1+2^nn\sigma\bigr)\,dx \\
\!\!\!& &\!\!\! -c(p,\beta,L)\mu^p \bigl(1+2^nn\sigma\bigr) r^n\\
\!\!\!&\geq&\!\!\! -C(n,p,\beta,L,\sigma)\int_{B_r\setminus B_{r/2}}\!|\nabla u_h(x)|^p\,dx-c(p,\beta,L)\mu^p \bigl(1+2^nn\sigma\bigr) r^n\\
\!\!\!&\geq&\!\!\! -C(n,p,\beta,L,\sigma,\mu)(\Theta+r^n)\,.
\end{eqnarray*}
 Thus,
from the above  estimates  we  conclude that
\begin{equation}\label{unonove}
I_{1,h}\geq-C(n,p,\beta,L,\sigma,\mu)(\Theta+r^n)\,.
\end{equation}
{\bf Step 4.}\,\, {\it Estimate of $I_{2,h}$.} We  use  the area formula for maps between rectifiable sets. To this aim,  for  $x\in\partial^*E_h$ denote by  $T_{h,x}:\pi_{h,x}\to\R^n$ the tangential differential at $x$ of  $\Phi$ along the approximate tangent space $\pi_{h,x}$ to $\partial^*E_h$, which is defined by  $T_{h,x}(\tau)=\nabla\Phi(x)\circ\tau$ for all $\tau\in\pi_{h,x}$. We recall (see \cite[Definition~2.68]{AFP}) that the $(n-1)$-dimensional jacobian of $T_{h,x}$ is given by
$$
J_{n-1}T_{h,x}=\sqrt{{\rm det}\bigl(T^*_{h,x}\circ T_{h,x}\bigr)}\,,
$$
where $T^*_{h,x}$ is the adjoint of the map $T_{h,x}$. To estimate $J_{n-1}T_{h,x}$, fix $x\in\partial^*E_h\cap(B_r\setminus B_{r/2})$. Denote by $\{\tau_1,\dots,\tau_{n-1}\}$ an orthonormal base for $\pi_{h,x}$, and by $L$ the $n\times(n-1)$ matrix representing $T_{h,x}$ with respect to the fixed base in $\pi_{h,x}$ and the standard base $\{e_1,\dots,e_n\}$ in $\R^n$. From \eqref{unocinque.1} we have
$$
L_{ij}=\nabla\Phi_i\cdot\tau_j=\frac{\varphi(|x|)}{|x|}e_i\cdot\tau_j+\Bigl(\varphi^\prime(|x|)-\frac{\varphi(|x|)}{|x|}\Bigr)\frac{x_i}{|x|}\frac{x\cdot\tau_j}{|x|},\qquad i=1,\dots,n,\,j=1,\dots,n-1\,.
$$
Thus, for $j,l=1,\dots,n-1$, we obtain
$$
(L^*\circ L)_{jl}=\frac{\varphi^2(|x|)}{|x|^2}\sum_{i=1}^n(e_i\cdot\tau_j)(e_i\cdot\tau_l)+\Bigl(\varphi^{\prime2}(|x|)-\frac{\varphi^2(|x|)}{|x|^2}\Bigr)\frac{(x\cdot\tau_j)(x\cdot\tau_l)}{|x|^2}\,.
$$
Since $J_{n-1}T_{h,x}$ is invariant by rotation, in order to evaluate det$(L^*\circ L)$ we may assume, without loss of generality, that $\tau_j=e_j$, for all $j=1,\dots,n-1$. We deduce that
$$
L^*\circ L=\frac{\varphi^2(|x|)}{|x|^2}I^{(n-1)}+\Bigl(\varphi^{\prime2}(|x|)-\frac{\varphi^2(|x|)}{|x|^2}\Bigr)\frac{x^\prime\otimes x^\prime}{|x|^2}\,,
$$
where $I^{(n-1)}$ denotes the identity map on $\R^{n-1}$ and $x^\prime=(x_1,\dots,x_{n-1})$. With a calculation similar to the one performed to obtain \eqref{unosei}, from the equality above we easily get that
$$
{\rm det}(L^*\circ L)=\Bigl(\frac{\varphi^2(|x|)}{|x|^2}\Bigr)^{n-1}\Bigl[1+\frac{|x|^2}{\varphi^2(|x|)}\Bigl(\varphi^{\prime2}(|x|)-\frac{\varphi^2(|x|)}{|x|^2}\Bigr)\frac{|x^\prime|^2}{|x|^2}\Bigr]
$$
and so, using \eqref{unocinque.1} we can estimate for $x\in\partial^*E_h\cap(B_r\setminus B_{r/2})$
\begin{eqnarray}\label{unodieci}
J_{n-1}T_{h,x}\!\!\!&=&\!\!\!\sqrt{{\rm det}(L^*\circ L)}=\Bigl(\frac{\varphi(|x|)}{|x|}\Bigr)^{n-1}\sqrt{1+\frac{|x|^2}{\varphi^2(|x|)}\Bigl(\varphi^{\prime2}(|x|)-\frac{\varphi^2(|x|)}{|x|^2}\Bigr)\frac{|x^\prime|^2}{|x|^2}} \\
\!\!\!&\leq&\!\!\! \Bigl(\frac{\varphi(|x|)}{|x|}\Bigr)^{n-2}\varphi^\prime(|x|)\leq\varphi^\prime(|x|)\leq1+\sigma+2^n(n-1)\sigma\,. \nonumber
\end{eqnarray}
To estimate $I_{2,h}$, we use the area formula for maps between rectifiable sets (\cite[Theorem~2.91]{AFP}), and we get
\begin{eqnarray*}
I_{2,h}\!\!\!&=&\!\!\!P(E_h,{\overline B}_r)-P({\widetilde E}_h,{\overline B}_r)=\int_{\partial^*E_h\cap{\overline B}_r}\!d\H^{n-1}-\int_{\partial^*E_h\cap{\overline B}_r}\!J_{n-1}T_{h,x}\,d\H^{n-1} \\
\!\!\!&=&\!\!\!\int_{\partial^*E_h\cap{\overline B}_r\setminus B_{r/2}}\!\left(1-J_{n-1}T_{h,x}\right)\,d\H^{n-1}+\int_{\partial^*E_h\cap B_{r/2}}\!\left(1-J_{n-1}T_{h,x}\right)\,d\H^{n-1}\,.
\end{eqnarray*}
Notice that the last integral in the above formula is nonnegative since $\Phi$ is a contraction in $B_{r/2}$, hence $J_{n-1}T_{h,x}<1$ in $B_{r/2}$, while from \eqref{unodieci} and \eqref{unouno} we have
$$
\int_{\partial^*E_h\cap{\overline B}_r\setminus B_{r/2}}\!\left(1-J_{n-1}T_{h,x}\right)\,d\H^{n-1}\geq-2^nnP(E_h,{\overline B}_r)\sigma\geq-2^nn\Theta\sigma\,,
$$
thus concluding that
\begin{equation}\label{unoundici}
I_{2,h}\geq-c(n)\Theta\sigma\,.
\end{equation}
{\bf Step 5.}\,\, {\it Estimate of $I_{3,h}$.} We recall \eqref{unodue}, \eqref{unotre}, \eqref{unosei} to obtain
\begin{eqnarray*}
I_{3,h}\!\!\!&=&\!\!\! \lambda_h\int_{E_h\cap B_r\setminus B_{r/2}}\!\left(J\Phi(x)-1\right)\,dx+\lambda_h\int_{E_h\cap B_{r/2}}\!\left(J\Phi(x)-1\right)\,dx \\
\!\!\!&\geq&\!\!\! \lambda_hC_1(n)\Bigl(\frac{\omega_n}{2^{n+2}}-\e\Bigr)\sigma r^n-\lambda_h\bigl[1-\bigl(1-(2^n-1)\sigma\bigr)^n\bigr]\e r^n \\
\!\!\!&\geq&\!\!\!\lambda_h\sigma r^n\Bigl[C_1(n)\frac{\omega_n}{2^{n+2}}-C_1(n)\e-(2^n-1)n\e\Bigr]\,.
\end{eqnarray*}
Therefore, if we choose $0<\e<\e(n)$, with $\e(n)$ depending only on the dimension, we have that
\begin{equation}\label{ih3}
I_{3,h}\geq\lambda_hC_2(n)\sigma r^n
\end{equation}
for some positive $C_2(n)$.

\medskip

\noindent {\bf Step 6.}\,\, {\it Conclusion of Step 1.}  Estimate \eqref{ih3}, together with \eqref{unoquattro}, \eqref{unonove} and \eqref{unoundici}, yields
$$
{\mathcal I}_{\lambda_h}(u_h,E_h)-{\mathcal I}_{\lambda_h}({\tilde u}_h,{\widetilde E}_h)\geq\lambda_h\sigma C_3r^n-C(n,p,\sigma,\mu)(\Theta+r^n) >0
$$
if $\lambda_h$ is sufficiently large. This contradicts the minimality of $(u_h,E_h)$, thus concluding the proof of the first part of Theorem \ref{uno}.

\medskip

\noindent{\bf Step 7.}
Conversely, if $(u,E)$ is a minimizer of $\mathcal{I}$ and $\lambda_0$ is as determined on Step 1, then for $\lambda>\lambda_0$ Steps 1--5 ensure the existence of a minimizer $(u_\lambda, E_\lambda)$  of $\mathcal{I}_\lambda$ with $|E_\lambda|=d$. Hence, by the minimality,
$$\mathcal{I}(u,E)\le \mathcal{I}(u_\lambda,E_\lambda)= \mathcal{I}_\lambda(u_\lambda,E_\lambda)\le \mathcal{I}_\lambda(u,E)=\mathcal{I}(u,E)$$
i.e.,
$$\mathcal{I}(u,E)= \mathcal{I}_\lambda(u_\lambda,E_\lambda)$$
and so $(u,E)$ also minimizes $\mathcal{I}_\lambda$.
\end{proof}

\bigskip

\section{Proof of Theorem \ref{duedue}}

\bigskip
 This section is devoted to the proof of our first regularity result, stated in Theorem  \ref{duedue}. The proof is obtained by establishing that the bulk energy and the perimeter of the free interface both decay on balls of radius $\rho$ as $\rho^{n-1}$, for  $\rho\to 0^+$. We divide it in two steps: In the first  we  prove the decay estimate for the perimeter, and in the second we address the decay of the bulk energies.

\begin{proof}[Proof of Theorem \ref{duedue}] Let $(u,E)$ be a solution of the problem  (P).

\medskip

\noindent
{\bf Step 1.}\,\,{\it First decay estimate.}  Fix $x_0\in \Omega$  and let $R\le \mathrm{dist}(x_0,\partial \Omega).$ Assume, without loss of generality that $0<R<1$.
 Here we want to prove that  there exists a constant $c_0=c_0(n,p,\lambda_0,\alpha,\beta,\ell,L)$ such that
 \begin{equation}\label{decay2a}
\int_{B_r(x_0)\cap E} |\nabla u|^p\,dx+P(E,B_r(x_0))\le  c_0 r^{n-1}\,,\end{equation}
for every $0<r<R$.

\noindent First, consider $x_0\in \partial E\cap\Omega$ and set $\widetilde E:= E\setminus B_r(x_0)$ where $0<r<R$.  For $\lambda_0$ determined in Theorem \ref{uno}, we have
$${\mathcal I_{\lambda_0}}(u,E)\le {\mathcal I_{\lambda_0}}(u,\widetilde E),$$
i.e.,
\begin{eqnarray*}
& &\int_\Om \Big(F(\nabla u)+\chi_{_{E}}G(\nabla u)\Big)\,dx+P(E,\Om)\cr\cr
&\le& \int_\Om \Big(F(\nabla u)+\chi_{_{\widetilde E}}G(\nabla u)\Big)\,dx+P(\widetilde E,\Om)\, + \lambda_0|\,|\widetilde E|\,-\,d\,|\,.
\end{eqnarray*}
Therefore,
\begin{equation*}
\int_\Om (\chi_{_{E}}- \chi_{_{\widetilde E}})G(\nabla u)\,dx+P(E,B_r(x_0))\le P(B_r(x_0))+ \lambda_0|\,|\widetilde E|\,-\,d\,|\,,
\end{equation*}
and so
\begin{equation*}
\int_{B_r(x_0)} \chi_{_{E}}G(\nabla u)\,dx+P(E,B_r(x_0))\le c(n) r^{n-1}+ c(n)\lambda_0r^n\le c(n,\lambda_0) r^{n-1}\,,\end{equation*}
since $r<1$. Lemma \ref{grobelow} implies that
\begin{equation*}
\alpha\frac{\ell}{2}\int_{B_r(x_0)} \chi_{_{E}}|\nabla u|^p\,dx-c(p,\mu,\alpha,\beta,\ell,L)|B_r(x_0)\cap E|+P(E,B_r(x_0))\le c(n,\lambda_0)  r^{n-1}\,,\end{equation*}
or, equivalently,
\begin{eqnarray*}
\alpha\frac{\ell}{2}\int_{B_r(x_0)} \chi_{_{E}}|\nabla u|^p\,dx+P(E,B_r(x_0))&\le& c(n,\lambda_0)  r^{n-1}+c(p,\mu,\alpha,\beta,\ell,L)|B_r(x_0)\cap E|\cr\cr
&\le& c(n,p,\mu,\alpha,\beta,\ell,L,\lambda_0)r^{n-1}\,.\end{eqnarray*}
Therefore
\begin{equation*}
\min\left\{\alpha\frac{\ell}{2},1\right\}\left[\int_{B_r(x_0)} \chi_{_{E}}|\nabla u|^p\,dx+P(E,B_r(x_0))\right]\le c(n,p,\mu,\alpha,\beta,\ell,L,\lambda_0)r^{n-1}\,.\end{equation*}
This inequality yields that
\begin{equation}\label{decay2}
\int_{B_r(x_0)\cap E} |\nabla u|^p\,dx+P(E,B_r(x_0))\le  c_0 r^{n-1}\,,\end{equation}
where we set $c_0:=c(n,p,\lambda_0,\alpha,\beta,\ell,L)$.

\noindent  If $x_0\not\in \partial E\cap\Omega$, or $B_r(x_0)\cap  E$ is not empty and we argue exactly as before, or $B_r(x_0)\subset \Omega\setminus E$ and  estimate \eqref{decay2a} is trivially satisfied.
\medskip

\noindent{\bf Step 2.}\,\,{\it Second decay estimate.}  Here we want to prove that there exist $\tau\in\left(0,\frac{1}{2}\right)$ and $\delta\in (0,1) $ such that for every $M>0$ there exists $h_0\in \mathbb{N}$ such that
  $\forall B(x_0,r)\subset\Omega$ we have
\begin{equation}\label{decay3} \int_{B_{r}(x_0)}|\nabla u|^p\,dx\le h_0 r^{n-1}\qquad \mathrm{or}\qquad
\int_{B_{\tau r}(x_0)}|\nabla u|^p\,dx\le M\tau^{n-\delta}\int_{B_{r}(x_0)}|\nabla u|^p\,dx\,.\end{equation}
In order to prove \eqref{decay3}, we argue by contradiction. Fix $ \tau\in (0,1/2)$ , $\delta\in (0,1) $ and choose $M>\tau^{\delta-n}$. Suppose that for every $h\in \mathbb{N}$, there exists a   ball  $ B_{ r_h}(x_h)\subset \Omega$ such that
\begin{equation}\label{decay04}\int_{B_{ r_h}(x_h)}|\nabla u|^p\,dx> h r_h^{n-1}\end{equation}
and
\begin{equation}\label{decay4}
\int_{B_{\tau r_h}(x_h)}|\nabla u|^p\,dx>M \tau^{n-\delta}\int_{B_{r_h}(x_h)}|\nabla u|^p\,dx\,.
\end{equation}
Note that estimates \eqref{decay2} and \eqref{decay04} yield
\begin{equation*}\int_{B_{r_h}(x_h)\cap E}|\nabla u|^p\,dx+P(E,B_{r_h}(x_h))\le c_0 r_h^{n-1}<\frac{c_0}{h}\int_{B_{ r_h}(x_h)}|\nabla u|^p\,dx\,,
\end{equation*}
and so
\begin{equation}\label{decay01}\int_{B_{r_h}(x_h)\cap E}|\nabla u|^p\,dx<\frac{c_0}{h}\int_{B_{ r_h}(x_h)}|\nabla u|^p\,dx\,.
\end{equation}
\noindent{\bf Substep 2.a.}\,\,{\it  Blow-up.} \qquad
Set
$$\varsigma_h^p=\medint_{B_{r_h}(x_h)}|\nabla u|^p\,dx$$
and,  for $y\in B_1(0)$,  introduce the sequence of rescaled functions defined as
$$v_h(y):=\frac{u(x_h+r_h y)-a_h}{\varsigma_h r_h},\,\,\,\,\mathrm{where}\,\,\,\,
a_h:=\medint_{B_{r_h}(x_h)} u(x)\,dx\,.$$
We have $$\nabla  v_h(y)=\frac{1}{\varsigma_h}\nabla  u(x_h+r_h y)$$
and a  change of variable yields
\begin{equation}\label{bl1}
\medint_{B_1}|\nabla v_h(y)|^p\,dy=\frac{1}{\varsigma_h^p}\medint_{B_{r_h}(x_h)}|\nabla u(x)|^p\,dx=1.
\end{equation}
Therefore, there exist a subsequence of $v_h$ (not relabeled) and $v\in W^{1,p}(B_1)$ such that
$$v_h\rightharpoonup v\qquad\mathrm{weakly\,\,in\,\,}W^{1,p}(B_1)\,, \,\,\mathrm{and}\, \, v_h\to v\qquad\mathrm{strongly\,\,in\,\,}L^{p}(B_1).$$
Moreover, the lower semicontinuity of the norm implies
\begin{equation}\label{norma1}
\medint_{B_1}|\nabla v(y)|^p\,dy\le \liminf_{h\to \infty}\medint_{B_1}|\nabla v_h(y)|^p\,dy=1.
\end{equation}
\noindent{\bf Substep 2.b.}\,\,{\it  We claim that $v_h\to v$ in $W^{1,p}_{loc}(B_1)$.} \qquad
Consider the sets
$$E^*_h:= \frac{E-x_h}{r_h}\cap B_1.$$
By \eqref{decay01}, and up to the extraction of a subsequence (not relabeled), $\chi_{_{E^*_h}}\to \chi_{_{E^*}}$ in $L^1$ (and weakly in $BV(B_1)$) for some set of finite perimeter $E^*\subset B_1$.

\noindent Using the minimality of $(u,E)$ with respect to $(u+\varphi,E)$, for $\varphi\in W^{1,p}_0(B_{r_h}(x_h))$ we obtain
$$\int_{B_{r_h}(x_h)}\!\!\!\Big(F(\nabla u(x))+\chi_{_{E}}G(\nabla u(x))\Big)\,dx\le \int_{B_{r_h}(x_h)}\!\!\!\Big(F(\nabla u(x)+\nabla \varphi(x))+\chi_{_{E}}G(\nabla u(x)+\nabla\varphi(x))\Big)\,dx,$$
or, equivalently, using the change of variable $x=x_h+r_h y$, we get
\begin{eqnarray}\label{decay5}& &\int_{B_1}\Big(F(\varsigma_h\nabla v_h(y))+\chi_{_{E^*_h}}G(\varsigma_h\nabla v_h(y))\Big)\,dy\cr\cr
&\le& \int_{B_1}\Big(F(\varsigma_h\nabla v_h(y)+\nabla \psi( y))+\chi_{_{E^*_h}}G(\varsigma_h\nabla v_h(y)+\nabla\psi ( y))\Big)\,dx\end{eqnarray}
for every $\psi\in W^{1,p}_0(B_1)$. Let $\eta\in C^\infty_0(B_1)$, $0\le \eta\le 1$.  Choosing $\psi_h(y)=\varsigma_h\eta (v-v_h)$ as test function in \eqref{decay5}, we get
\begin{eqnarray}\label{decay611}& &\int_{B_1}\Big(F(\varsigma_h\nabla v_h(y))+\chi_{_{E^*_h}}G(\varsigma_h\nabla v_h(y))\Big)\,dy\cr\cr
&\le& \int_{B_1}\Big(F\big(\varsigma_h\eta\nabla v( y)+ \varsigma_h(1-\eta)\nabla v_h( y)+\nabla\eta\varsigma_h(v-v_h)\big)\Big)\,dy\cr\cr
&+&\int_{B_1}\chi_{_{E^*_h}}\Big(G\big(\varsigma_h\eta\nabla v( y)+ \varsigma_h(1-\eta)\nabla v_h( y)+\nabla\eta\varsigma_h(v-v_h)\big)\Big)\,dy\cr\cr
&\le& \int_{B_1}\Big(F\big(\varsigma_h\eta\nabla v( y)+ \varsigma_h(1-\eta)\nabla v_h( y)\big)\Big)\,dy+\int_{B_1}\chi_{_{E^*_h}}\Big(G\big(\varsigma_h\eta\nabla v( y)+ \varsigma_h(1-\eta)\nabla v_h( y)\big)\Big)\,dy\cr\cr
&+&\int_{B_1}\Big\langle D_\xi F\big(\varsigma_h\eta\nabla v( y)+ \varsigma_h(1-\eta)\nabla v_h( y)+\nabla\eta\varsigma_h(v-v_h)\big),\nabla\eta\varsigma_h(v-v_h)\Big\rangle\,dy\cr\cr
&+&\int_{B_1}\chi_{_{E^*_h}}\Big\langle D_\xi G\big(\varsigma_h\eta\nabla v( y)+ \varsigma_h(1-\eta)\nabla v_h( y)+\nabla\eta\varsigma_h(v-v_h)\big),\nabla\eta\varsigma_h(v-v_h)\Big\rangle\,dy\cr\cr
&\le& \int_{B_1}\Big(F\big(\varsigma_h\eta\nabla v( y)+ \varsigma_h(1-\eta)\nabla v_h( y)\big)\Big)\,dy+\int_{B_1}\chi_{_{E^*_h}}\Big(G\big(\varsigma_h\eta\nabla v( y)+ \varsigma_h(1-\eta)\nabla v_h( y)\big)\Big)\,dy\cr\cr
&+&c\int_{B_1}\big(\mu^2+|\varsigma_h\nabla v_h|^2+|\varsigma_h\nabla v|^2+|\varsigma_h(v-v_h)|^2\big)^{\frac{p-1}{2}}|\varsigma_h(v-v_h)|\,dy
\end{eqnarray}
where, in the last inequality, we used  \eqref{(H4)}. Hence, using  H\"older's inequality and the convexity of $F$ and $G$  in estimate \eqref{decay611}, we obtain
\begin{eqnarray}\label{decay6}
& &\int_{B_1}\Big(F(\varsigma_h\nabla v_h(y))+\chi_{_{E^*_h}}G(\varsigma_h\nabla v_h(y))\Big)\,dy\nonumber\\
&\le&\int_{B_1}\Big((1-\eta)F(\varsigma_h\nabla v_h( y))\,dy+\eta F(\varsigma_h\nabla v( y))\Big)\,dy\nonumber\\
&+&
\int_{B_1}\chi_{_{E^*_h}}\Big((1-\eta)G(\varsigma_h\nabla v_h( y))+\eta G(\varsigma_h\nabla v( y))\Big)\,dy\\
&+&c\int_{B_1}|\varsigma_h(v-v_h)|^p\,dx+\varsigma_h^p\left(\int_{B_1}\mu^p+|\nabla v_h|^p+|\nabla v|^p\,dx\right)^{\frac{p-1}{p}}\!\!\left(\int_{B_1}|v-v_h|^p\,dx\right)^{\frac{1}{p}}\!\!, \nonumber\end{eqnarray}
since we may suppose that $\varsigma_h>1$ for $h$ large. In fact, by \eqref{decay04} and the definition of $\varsigma_h$, we have
\begin{equation}\label{sigma}\varsigma_h^p\ge \frac{h}{r_h},\end{equation} and so $\varsigma_h\to +\infty$ as $h\to +\infty$.
By virtue of  \eqref{bl1}, from estimate \eqref{decay6} we infer that
\begin{eqnarray}\label{decay60}
&  &\int_{B_1}\eta\Big(F(\varsigma_h\nabla v_h(y))+\chi_{_{E^*_h}}G(\varsigma_h\nabla v_h(y))\Big)\,dy\cr\cr
&\le& \int_{B_1}\Big(\eta F(\varsigma_h\nabla v( y))\Big)\,dy+\int_{B_1}\chi_{_{E^*_h}}\Big(\eta G(\varsigma_h\nabla v( y))\Big)\,dy\cr\cr
&+&c\varsigma_h^p\int_{B_1}|v-v_h|^p+c\varsigma_h^p\left(\int_{B_1}|v-v_h|^p\,dx\right)^{\frac{1}{p}}.
\end{eqnarray}
Note that, by changing variable in \eqref{decay01}, we have
$$
r_h^n\int_{B_1\cap E^*_h}|\nabla u (x_h+r_h y)|^p\,dy< \frac{c_0}{h}r_h^n\int_{B_1}|\nabla u (x_h+r_h y)|^p\,dy\,,$$
and thus, by the definition of $v_h$,
$$\int_{B_1\cap E^*_h}|\nabla v_h ( y)|^p\,dy<  \frac{c_0}{h}\int_{B_1}|\nabla v_h ( y)|^p\,dy$$
and, by the use of \eqref{bl1}, we get
\begin{equation}\label{bl2} \int_{B_1\cap E^*_h}|\nabla v_h ( y)|^p\,dy<  \frac{c_0\omega_n}{h} .\end{equation}
Since  $\chi_{_{E^*_h}}\to \chi_{_{E^*}}$ weakly in $BV(B_1)$, by Fatou's Lemma and \eqref{bl2} we obtain
\begin{equation}\label{decay7}
\int_{B_1\cap E^*}|\nabla v( y)|^p\,dy\le \liminf_h\int_{B_1\cap E^*_h}|\nabla v_h ( y)|^p\,dy=0.
\end{equation}
Using assumption  (G1) and the homogeneity of $F$ in \eqref{decay60}, we get
\begin{eqnarray*}
\int_{B_1}\eta \varsigma^{p}_h F(\nabla v_h(y))\,dy&\le&
\int_{B_1}\eta \varsigma_h^pF(\nabla v( y))\,dy+c\int_{B_1}\chi_{_{E^*_h}}\big( |\varsigma_h\nabla v( y)|^p+|\varsigma_h\nabla v_h( y)|^p\big)\,dy\cr\cr
&+&c\varsigma_h^p\int_{B_1}|v-v_h|^p\,dy+c\varsigma_h^p\left(\int_{B_1}|v-v_h|^p\,dy\right)^{\frac{1}{p}},
\end{eqnarray*}
i.e.,
\begin{eqnarray}\label{decay61}
\int_{B_1}\eta F(\nabla v_h(y))\,dy&\le&
\int_{B_1}\eta F(\nabla v( y))\,dy+\int_{B_1}\chi_{_{E^*_h}}( |\nabla v( y)|^p+|\nabla v_h( y)|^p)\,dy\cr\cr
&+&c\int_{B_1}|v-v_h|^p\,dy+c\left(\int_{B_1}|v-v_h|^p\,dy\right)^{\frac{1}{p}}.
\end{eqnarray}
Passing to the limit as $h\to +\infty$ in \eqref{decay61}, by virtue of \eqref{decay7}, the strong convergence of $v_h$ to $v$ in $L^p$ and the lower semicontinuity of $F$, we obtain
$$\int_{B_1}\eta F(\nabla v(y))\,dy\le \liminf_h \int_{B_1}\eta F(\nabla v_h(y))\,dy\le \int_{B_1}\eta F(\nabla v( y))\,dy,$$
that is, \begin{equation}\label{strong}\lim_h \int_{B_1}\eta F(\nabla v_h(y))\,dy=\int_{B_1}\eta F(\nabla v(y))\,dy\,.\end{equation}
By the strong p-convexity of $F$ and Lemma \ref{V}, we have
\begin{eqnarray}\label{strong1}& \int_{B_1}\eta |V(\nabla v_h(y))-V(\nabla v(y))|^2\,dy\\
&\le c(p,\ell)\int_{B_1}\eta \Big(F(\nabla v_h(y))-F(\nabla v(y))\Big)-\langle D_\xi F(\nabla v_h(y)),\eta(\nabla v_h(y)-\nabla v(y))\rangle\,dy.\nonumber
\end{eqnarray}
By the minimality of $(u,E)$, we get
$$\int_{B_{r_h}(x_h)}\Big\langle D_\xi F(\nabla u(y))+\chi_{_{E}}D_\xi G(\nabla u(y)),\nabla \varphi\Big\rangle\, dx=0$$
for every $\varphi\in W^{1,p}_0(B_{r_h}(x_h))$ or, equivalently,
$$\int_{B_1}\Big\langle D_\xi F(\varsigma_h\nabla v_h(y))+\chi_{_{E^*_h}}D_\xi G(\varsigma_h\nabla v_h(y)),\nabla \psi\Big\rangle\, dx=0$$
for every $\psi\in W^{1,p}_0(B_1)$, or still
\begin{equation}\label{strong2}
\int_{B_1}\Big\langle D_\xi F(\varsigma_h\nabla v_h(y)),\nabla\psi\Big\rangle\,dy=-\int_{B_1}\Big\langle \chi_{_{E^*_h}}D_\xi G(\varsigma_h\nabla v_h(y)),\nabla \psi\Big\rangle\, dx\,.
\end{equation}
Then, choosing $\psi:=\eta(v_h-v)$ with $\eta\in C^\infty_0(B_1)$ as test function in \eqref{strong2}, we obtain
\begin{eqnarray}\label{strong3}
& &\int_{B_1}\Big\langle D_\xi F(\varsigma_h\nabla v_h(y)),\eta(\nabla v_h-\nabla v)\Big\rangle\,dy\cr\cr
&=&-\int_{B_1}\Big\langle D_\xi F(\varsigma_h\nabla v_h(y)),\nabla \eta( v_h- v)\Big\rangle\,dy\cr\cr
& &-\int_{B_1}\chi_{_{E^*_h}}\Big\langle D_\xi G(\varsigma_h\nabla v_h(y)),\eta(\nabla v_h-\nabla v)\Big\rangle\,dy\cr\cr
& &-\int_{B_1}\chi_{_{E^*_h}}\Big\langle D_\xi G(\varsigma_h\nabla v_h(y)),\nabla\eta( v_h- v)\Big\rangle\,dy\,.
\end{eqnarray}
Using estimates \eqref{(H4)} for $D_\xi F$ and $D_\xi G$,  \eqref{strong3} yields
 \begin{eqnarray}\label{strong30}
& &\left|\int_{B_1}\Big\langle D_\xi F(\varsigma_h\nabla v_h(y)),\eta(\nabla v_h-\nabla v)\Big\rangle\,dy\right|\cr\cr
&\le & c(p,\beta,L)\int_{B_1}|\varsigma_h\nabla v_h(y)|^{p-1}|\nabla \eta|| v_h- v|\,dy\cr\cr
&+&c(p,\beta,L) \int_{B_1\cap E^*_h}|\varsigma_h\nabla v_h(y)|^{p-1}|\eta||\nabla v_h-\nabla v|\,dy.
\end{eqnarray}
By the homogeneity of $F$, H\"older's inequality, \eqref{bl1} and \eqref{bl2},  \eqref{strong30} implies that
\begin{eqnarray}\label{strong4}
& &\left|\int_{B_1}\Big\langle D_\xi F(\nabla v_h(y)),\eta(\nabla v_h-\nabla v)\Big\rangle\,dy\right|\cr\cr
&\le & c(p,\beta,L,||\nabla \eta||_\infty)\left(\int_{B_1}|\nabla v_h(y)|^{p}\,dy\right)^{\frac{p-1}{p}}\left(\int_{B_1}| v_h- v|^p\,dy\right)^{\frac{1}{p}}\cr\cr
&+&c(p,\beta,L,|| \eta||_\infty) \left(\int_{B_1\cap E^*_h}|\nabla v_h(y)|^{p}\,dy\right)^{\frac{p-1}{p}}\left(\int_{B_1}|\nabla v_h|^p+|\nabla v|^p\,dy\right)^{\frac{1}{p}}\cr\cr
&\le & c(n,p,\beta,L,||\nabla \eta||_\infty)\left(\int_{B_1}| v_h- v|^p\,dy\right)^{\frac{1}{p}}+c(p,\beta,L,|| \eta||_\infty) \left(\frac{c_0}{h}\right)^{\frac{p-1}{p}}.
\end{eqnarray}
Since $v_h$ converge strongly to $v$ in $L^p(B_1)$, passing to the limit as $h\to \infty$ in \eqref{strong4}, we get
\begin{equation}\label{strong5}\lim_{h\to+\infty}\left|\int_{B_1}\langle D_\xi F(\nabla v_h(y)),\eta(\nabla v_h-\nabla v)\rangle\,dy\right|=0.\end{equation}
Passing to the limit in \eqref{strong1} and using \eqref{strong} and \eqref{strong5}, we obtain
\begin{equation*} \lim_{h\to+\infty}\int_{B_1}\eta |V(\nabla v_h(y))-V(\nabla v(y))|^2\,dy=0,
\end{equation*}
which, by Lemma \ref{V},   implies that
\begin{equation}\label{strong7} \lim_{h\to+\infty}\int_{B_1}\eta(\mu^2+|\nabla v_h(y)|^2+|\nabla v(y)|^2)^{\frac{p-1}{2}}|\nabla v_h(y)-\nabla v(y)|^2\,dy=0.
\end{equation}
In the case $p\ge 2$, one can easily check that \eqref{strong7} implies
 $$v_h\to v\qquad\mathrm{strongly\,\,in\,\,}W^{1,p}_{\mathrm{loc}}(B_1)\,.$$
 In the case $1<p<2$, it suffices to observe that  H\"older's inequality with exponents $\frac{2}{p}$ and $\frac{2}{2-p}$ yields
\begin{eqnarray}& &\int_{B_1}\!\eta|\nabla v_h-\nabla v|^p\,dx\cr\cr
&\le&\left(\int_{B_1}\eta(\mu^2+|\nabla v_h|^2+|\nabla v|^2)^{\frac{p-2}{2}}|\nabla v_h-\nabla v|^2\,dx\right)^{\frac{p}{2}}\left(\int_{B_1}\eta(\mu^2+|\nabla v_h|^2+|\nabla v|^2)^{\frac{p}{2}}\,dx\right)^{\frac{2-p}{2}}\cr\cr
&\le&c\left(\int_{B_1}\eta(\mu^2+|\nabla v_h|^2+|\nabla v|^2)^{\frac{p-2}{2}}|\nabla v_h-\nabla v|^2\,dx\right)^{\frac{p}{2}}\,,\end{eqnarray}
where we used  \eqref{bl1}.  Hence, also in this case, by \eqref{strong7} we conclude that
$$v_h\to v\qquad\mathrm{strongly\,\,in\,\,}W^{1,p}_{\mathrm{loc}}(B_1)\,,$$
and this asserts the claim.
\medskip

\noindent {\bf Substep 2.c.}\,\,{\it Reaching a contradiction.} Notice that \eqref{decay4} can be written as
\begin{equation*}
\medint_{B_{\tau r_h}(x_h)}|\nabla u|^p\,dx> M\tau^{-\delta}\medint_{B_{r_h}(x_h)}|\nabla u|^p\,dx,
\end{equation*}
or, equivalently,
\begin{equation}\label{conc2}
\frac{1}{\varsigma_h^p}\medint_{B_{\tau r_h}(x_h)}|\nabla u|^p\,dx> M\tau^{-\delta}\,,
\end{equation}
by the definition of $\varsigma_h$. By the change of variable $x=x_h+r_h y$ and the definition of $v_h$, from \eqref{conc2} we infer that
\begin{equation*}
\frac{1}{\varsigma_h^p}\medint_{B_{\tau }}|\varsigma_h\nabla v _h|^p\,dx> M\tau^{-\delta},
\end{equation*}
i.e.,
\begin{equation}\label{conc30}
\medint_{B_{\tau }}|\nabla v _h|^p\,dx>M \tau^{-\delta}.
\end{equation}
By virtue of the strong convergence of $v_h$ to $v$ in $W^{1,p}_{\mathrm{loc}}(B_1)$ and \eqref{norma1}, we have that
\begin{equation}\label{conc4}
\lim_h \medint_{B_{\tau }}|\nabla v _h|^p= \medint_{B_{\tau }}|\nabla v |^p\le \frac{1}{\tau^n}\,.
\end{equation}
 Clearly, \eqref{conc4} contradicts \eqref{conc30} because $M>\tau^{\delta-n}$.

 \medskip

\noindent {\bf Step 3.}\,\,{\it  Conclusion.}\,\,
We conclude that if $(u,E)$ is a solution of (P), then  there exist $\tau\in\left(0,\frac{1}{2}\right)$ and $\delta\in (0,1)$ such that, setting $M=1$, there exists $h_0\in \mathbb{N}$ with the property that whenever   $B_r(x)\subset \Omega$, then
\begin{equation*}
\int_{B_{ r}(x_0)}|\nabla u|^p\,dx\le h_0r^{n-1}\qquad \mathrm{or}\qquad \int_{B_{\tau r}(x_0)}|\nabla u|^p\,dx\le \tau^{n-\delta}\int_{B_{r}(x_0)}|\nabla u|^p\,dx\,.
\end{equation*}
Hence,
\begin{equation*}
\int_{B_{\tau r}(x_0)}|\nabla u|^p\,dx\le \tau^{n-\delta}\int_{B_{r}(x_0)}|\nabla u|^p\,dx
+h_0r^{n-1} \,,
\end{equation*}
and using Lemma \ref{iter} with $\varphi(\rho):=\int_{B_{\rho}(x_0)}|\nabla u|^p\,dx$, $\gamma=n-\delta$ and $\beta=n-1$, we obtain that
 \begin{equation*}
\int_{B_{ \rho}(x_0)}|\nabla u|^p\,dx\le c\left\{ \left(\frac{\rho}{r}\right)^{n-1}\int_{B_{ r}(x_0)}|\nabla u|^p\,dx
+h_0\rho^{n-1} \right\}\,,
\end{equation*}
for every $0<\rho<r\le R$, and so
 \begin{equation*}
\int_{B_{ \rho}(x_0)}|\nabla u|^p\,dx\le C \rho^{n-1}\,.
\end{equation*}
By H\"older's inequality
$$\int_{B_{ \rho}(x_0)}|\nabla u|\,dx\le c\left(\int_{B_{ \rho}(x_0)}|\nabla u|^p\,dx\right)^{\frac{1}{p}}  \rho^{\frac{n}{p'}}\le C \rho^{\frac{n-1}{p}+\frac{n}{p'}}= C\rho^{n-\frac{1}{p}}\,.$$
  Theorem \ref{Morrey} yields that, at least, $u$ is locally H\"older continuous with exponent $\frac{1}{p'}$. The regularity of the boundary at this point can be obtained arguing as in \cite[Theorem 2.2]{AB}, with the obvious modifications.
\end{proof}

\section{Proof of Theorem \ref{due}  -- Full regularity}

\bigskip

This section is devoted to the proof of the full regularity result stated in Theorem \ref{due}. The key point is to prove that if the ratio $\frac{\beta}{\alpha+1}$, where $\alpha$ and $\beta$ are the parameters appearing in hypotheses (G1) and (G2), is sufficiently small  then $\int_{B_\rho}|\nabla u|^p$ decays as $\rho^{n-1+\delta}$.

\begin{proof}[Proof of Theorem~\ref{due}]\,\,{\bf Step 1.} Let $(u,E)$ be a minimal configuration of problem (P). We first show that $u\in C^{0,1/p'+\delta}_{\mathrm{loc}}(\Omega)$ for some positive $\delta$, with $p'=\frac{p}{p-1}$. Fix $x\in\Omega$ and a ball $B_r(x)\subset\!\subset\Om$. Assume, without  loss of generality, that $x=0$ and $r<1$. In what follows, we will omit the dependence on the center simply denoting by $B_r$ the ball $B_r(0)$.
By Theorem \ref{uno}, we have that $(u,\,E)$  is a minimizer of  problem \eqref{unozero} for $\lambda$ sufficiently large.
 Let  $v$ be the minimizer of $$w\in W^{1,p}(B_r)\mapsto\displaystyle{\int_{B_r}(F+G)(\nabla w)\,dx},$$  satisfying the boundary condition  $v=u$ on $\partial B_r$.  Then \begin{equation}\label{ELu}
\int_{B_r}\!\,\Big(D_\xi F(\nabla u)+\chi_{_{E}}D_\xi G(\nabla u)\Big)\cdot \nabla\varphi\, dx\,=\,0
\end{equation}
and
 \begin{equation}\label{ELv}
\int_{B_r}D_\xi(F+G)(\nabla v)\cdot \nabla\varphi \,dx\,=\,0
\end{equation}
for all $\varphi\in W^{1,p}_0(B_r)$.
Note that assumptions  (F1)-(F2) and (G1)-(G2) imply that the integrand $F+G$ satisfies
$$
0\le (F+G)(\xi)\le \widetilde L (\mu^2+| \xi|^2 )^{\frac{p}{2}},   \leqno{\rm (H1)}
$$
$$
\int_{\Omega}(F+G)(\xi+\nabla \varphi)\, dx \ge \int_{\Omega}\Big( (F+G)(\xi) +\widetilde\ell (\mu^2+| \xi|^2 +|\nabla\varphi|^2)^{\frac{p-2}{2}}|\nabla\varphi|^2\Big)\,dx,\leqno{\rm (H2)}
$$
and (see \eqref{deriv})
$$\langle D_\xi (F+G)(\xi)-D_\xi (F+G)(\eta),\xi-\eta\rangle\ge c(p)\widetilde\ell |V(\xi)-V(\eta)|^2,\leqno{\rm (H3)}$$
with growth and coercivity  constants $\widetilde L$, $\widetilde \ell$ such that
$$\tilde L\le (\beta+1) L\qquad \mathrm{and} \qquad \tilde \ell\ge (\alpha+1)\ell\,.$$
By virtue of (H1) and (H2), we can apply Theorem \ref{fofu} and \eqref{cost} to $H=F+G$, to obtain that for all $0<\ro<\frac{{r}}{2}$
\begin{eqnarray}\label{sup}
\int_{B_\ro}\!(\mu^2+|\nabla v|^2)^\frac{p}{2}\,dx&\leq & |B_\ro|\,\sup_{B_\ro}(\mu^2+|\nabla v|^2)^\frac{p}{2}\le c_n \ro^n \, \sup_{B_\frac{r}{2}}(\mu^2+|\nabla v|^2)^\frac{p}{2}\cr\cr
&\le & c \left(\frac{\widetilde L}{\widetilde \ell}\right)^{\sigma} \Bigl(\frac{\ro}{r}\Bigr)^n\int_{B_r}(\mu^2+|\nabla v|^2)^\frac{p}{2}\,dx\,,
\end{eqnarray}
for some constants $c=c(n,p)\ge 1$ and $\sigma=\frac{2n}{p}$. On the other hand if $\frac{{r}}{2}\le \ro<r$, one easily gets that
$$\int_{B_\ro}\!(\mu^2+|\nabla v|^2)^\frac{p}{2}\,dx\leq 2^n\frac{\ro^n}{r^n} \int_{B_r}(\mu^2+|\nabla v|^2)^\frac{p}{2}\,dx\,.$$
Therefore estimate \eqref{sup} holds for every $0<\ro<r$.
Subtracting  \eqref{ELv}  from  \eqref{ELu}, we obtain

\begin{equation*}
\int_{B_r}\!\Big(D_\xi (F+G)(\nabla u)-D_\xi (F+G)(\nabla v)\Big)\cdot \nabla\varphi \, dx -\int_{B_r\setminus E}\!D_\xi G(\nabla u)\cdot \nabla\varphi \, dx\, =0,
\end{equation*}
or, equivalently,

\begin{equation}\label{diff2}
\int_{B_r}\!\Big(D_\xi (F+G)(\nabla u)-D_\xi (F+G)(\nabla v)\Big)\cdot \nabla\varphi \, dx =\int_{B_r\setminus E}\!D_\xi G(\nabla u)\cdot \nabla\varphi \, dx\,.
\end{equation}
Next, we  treat separately the cases $p\ge 2$ and $1<p<2$.

\bigskip

\noindent {\it Case $p\ge 2$.} \qquad
Set  $\varphi :=u-v$ in  \eqref{diff2}. In \eqref{diff2} we use  (H3) and Lemma \ref{V} in the left hand side, the second condition in \eqref{(H4)} and H\"older's inequality in the right hand side, thus obtaining
\begin{eqnarray*}
&&c_{p,n}\widetilde\ell\int_{B_r}\!(\mu^2+|\nabla u|^2+|\nabla v|^2)^{\frac{p-2}{2}}|\nabla u-\nabla v|^2\,dx\leq c_p\beta L\int_{B_r}\!(\mu^2+|\nabla u|^2)^{\frac{p-1}{2}}|\nabla u-\nabla v|\,dx\cr\cr
&\le&c_p\beta L\left(\int_{B_r}\!(\mu^2+|\nabla u|^2)^{\frac{p}{2}}\, dx\right)^{\frac{1}{2}}\left(\int_{B_r}\!(\mu^2+|\nabla u|^2)^{\frac{p-2}{2}}|\nabla u-\nabla v|^2\,dx\right)^{\frac{1}{2}}\cr\cr
&\le &c_p\beta L\left(\int_{B_r}\!(\mu^2+|\nabla u|^2)^{\frac{p}{2}}\, dx\right)^{\frac{1}{2}} \left(\int_{B_r}\!(\mu^2+|\nabla u|^2+|\nabla v|^2)^{\frac{p-2}{2}}|\nabla u-\nabla v|^2\,dx\right)^{\frac{1}{2}}
\end{eqnarray*}
where, in the last inequality, we used that $p\ge 2$.
Hence
\begin{equation}\label{diff4}
\int_{B_r}\!(\mu^2+|\nabla u|^2+|\nabla v|^2)^{\frac{p-2}{2}}|\nabla u-\nabla v|^2\,dx\leq c_{p,n} \left(\frac{\beta L}{\widetilde\ell
}\right)^{2}\int_{B_r}\!(\mu^2+|\nabla u|^2)^{\frac{p}{2}}\, dx\, .
\end{equation}
By virtue of \eqref{diff4}, one  has that for $0<\rho<r$
\begin{eqnarray}\label{diff40}\int_{B_\ro}\!|\nabla u-\nabla v|^p\,dx&\le& c_p\int_{B_\ro}(\mu^2+|\nabla u|^2+|\nabla v|^2)^{\frac{p-2}{2}}|\nabla u-\nabla v|^2\,dx\cr\cr
&\leq& c_{p,n}\left(\frac{\beta L}{\widetilde\ell
}\right)^{2}\int_{B_r}\!(\mu^2+|\nabla u|^2)^{\frac{p}{2}}\, dx\, ,\end{eqnarray}
 therefore, from \eqref{sup} and \eqref{diff40}, we  get
\begin{eqnarray}\label{diff5}
\biggl(\int_{B_\ro}\!|\nabla u|^p\,dx\biggr)^{1/p}\!\!\!&\leq&\!\!\!\biggl(\int_{B_\ro}\!|\nabla u-\nabla v|^p\,dx\biggr)^{1/p}+\biggl(\int_{B_\ro}\!|\nabla v|^p\,dx\biggr)^{1/p} \cr\cr
 \!\!\!&\leq& c_{n,p} \left(\frac{\beta L}{\widetilde\ell
}\right)^{\frac{2}{p}}\left(\int_{B_r}\!(\mu^2+|\nabla u|^2)^{\frac{p}{2}}\, dx\right)^{1/p}\cr\cr
&+&c_{n,p}\left(\frac{\widetilde L}{\widetilde\ell}\right)^{\frac{\sigma}{p}}\Bigl(\frac{\ro}{r}\Bigr)^{n/p}
\biggl(\int_{B_r}\!(\mu^2+|\nabla v|^2)^{\frac{p}{2}}\,dx\biggr)^{1/p}.
\end{eqnarray}
By  Lemma \ref{grobelow} applied for $H=F+G$ and by the minimality of $v$, we have
\begin{eqnarray}\label{due4}
\int_{B_r}\!(\mu^2+|\nabla v|^2)^{\frac{p}{2}}\,dx&\leq&
 \frac{2}{\widetilde \ell}\int_{B_r}\!(F+G)(\nabla v)\,dx+c(n,p,\widetilde L,\widetilde \ell,\mu)r^n \cr\cr
  &\leq& \frac{2}{\widetilde \ell}\int_{B_r}\!(F+G)(\nabla u)\,dx+c(n,p,\widetilde L,\widetilde \ell,\mu)r^n \cr\cr
&\leq& \frac{2\widetilde L}{\widetilde \ell} \int_{B_r}(\mu^2+|\nabla u|^2)^{\frac{p}{2}}\,dx +c(n,p,\widetilde L,\widetilde \ell,\mu)r^n\;,
\end{eqnarray}
where in last line we used the growth assumption (H1). Combining \eqref{diff5} and \eqref{due4}, we obtain, for all $0<\ro<r$, that
\begin{eqnarray*}
\left(\int_{B_\ro}\!|\nabla u|^p\,dx\right)^{\frac{1}{p}}&\leq& c_{n,p} \left(\frac{\beta L}{\widetilde\ell
}\right)^{\frac{2}{p}}\left(\int_{B_r}\!(\mu^2+|\nabla u|^2)^{\frac{p}{2}}\, dx\right)^{\frac{1}{p}}\cr\cr
&+&c_{n,p}
\left(\frac{\widetilde L}{\widetilde\ell}\right)^{\frac{\sigma+1}{p}}\Bigl(\frac{\ro}{r}\Bigr)^{\frac{n}{p}}
\left(\int_{B_r}(\mu^2+|\nabla u|^2)^{\frac{p}{2}}\,dx\right)^{\frac{1}{p}}+
c(n,p,\widetilde L,\widetilde \ell,\mu)\ro^{\frac{n}{p}}\cr\cr
&=&c_{n,p}\left[\left(\frac{\beta L}{\widetilde\ell
}\right)^{\frac{2}{p}}+
\left(\frac{\widetilde L}{\widetilde\ell}\right)^{\frac{\sigma+1}{p}}\Bigl(\frac{\ro}{r}\Bigr)^{\frac{n}{p}}\right]
\left(\int_{B_r}(\mu^2+|\nabla u|^2)^{\frac{p}{2}}\,dx\right)^{\frac{1}{p}}\cr\cr
&+&
c(n,p,\widetilde L,\widetilde \ell,\mu)\ro^{\frac{n}{p}}\;,
\end{eqnarray*}
and therefore the following estimate holds
\begin{eqnarray}\label{due4bis}
\int_{B_\ro}\!|\nabla u|^p\,dx&\leq& c_{n,p}\left[\left(\frac{\beta L}{\widetilde\ell
}\right)^{\frac{2}{p}}+
\left(\frac{\widetilde L}{\widetilde\ell}\right)^{\frac{\sigma+1}{p}}\Bigl(\frac{\ro}{r}\Bigr)^{\frac{n}{p}}\right]^p
\int_{B_r}(\mu^2+|\nabla u|^2)^{\frac{p}{2}}\,dx\cr\cr
&+&c(n,p,\widetilde L,\widetilde \ell,\mu)\ro^n\;.
\end{eqnarray}
for every $0<\rho<r$.
\bigskip

\noindent {\it Case $1<p< 2$.} \qquad
As before, we choose  $\varphi =u-v$ in  \eqref{diff2}. In  \eqref{diff2} we use  (H3) and Lemma \ref{V} in the left hand side, the second condition in \eqref{(H4)} and H\"older's inequality in the right hand side to obtain
\begin{eqnarray*}
&&c_{n,p}\widetilde\ell\int_{B_r}\!(\mu^2+|\nabla u|^2+|\nabla v|^2)^{\frac{p-2}{2}}|\nabla u-\nabla v|^2\,dx\leq c_p\beta L\int_{B_r}\!(\mu^2+|\nabla u|^2)^{\frac{p-1}{2}}|\nabla u-\nabla v|\,dx\cr\cr
&\leq& c_p\beta L\int_{B_r}\!(\mu^2+|\nabla u|^2+|\nabla v|^2)^{\frac{p-1}{2}}|\nabla u-\nabla v|\,dx\cr\cr
&\le&c_p\beta L\left(\int_{B_r}\!(\mu^2+|\nabla u|^2+|\nabla v|^2)^{\frac{p}{2}}\, dx\right)^{\frac{1}{2}}\left(\int_{B_r}\!(\mu^2+|\nabla u|^2+|\nabla v|^2)^{\frac{p-2}{2}}|\nabla u-\nabla v|^2\,dx\right)^{\frac{1}{2}}
\end{eqnarray*}
and so
\begin{equation}\label{diff444}
\int_{B_r}\!(\mu^2+|\nabla u|^2+|\nabla v|^2)^{\frac{p-2}{2}}|\nabla u-\nabla v|^2\,dx\leq c_{n,p} \left(\frac{\beta L}{\widetilde\ell
}\right)^{2}\int_{B_r}\!(\mu^2+|\nabla u|^2+|\nabla v|^2)^{\frac{p}{2}}\, dx\, .
\end{equation}
On the other hand, for $1<p<2$, H\"older's inequality with exponents $\frac{2}{p}$ and $\frac{2}{2-p}$ yields
\begin{eqnarray}\label{diff401}& &\int_{B_\ro}\!|\nabla u-\nabla v|^p\,dx\cr\cr
&\le&\left(\int_{B_\ro}(\mu^2+|\nabla u|^2+|\nabla v|^2)^{\frac{p-2}{2}}|\nabla u-\nabla v|^2\,dx\right)^{\frac{p}{2}}\left(\int_{B_\ro}(\mu^2+|\nabla u|^2+|\nabla v|^2)^{\frac{p}{2}}\,dx\right)^{\frac{2-p}{2}}\cr\cr
&\le&c_{n,p}\left(\frac{\beta L}{\widetilde\ell
}\right)^{p}\left(\int_{B_r}\!(\mu^2+|\nabla u|^2+|\nabla v|^2)^{\frac{p}{2}}\, dx\right)^{\frac{p}{2}}\left(\int_{B_\ro}(\mu^2+|\nabla u|^2+|\nabla v|^2)^{\frac{p}{2}}\,dx\right)^{\frac{2-p}{2}}\cr\cr
&=&c_{n,p}\left(\frac{\beta L}{\widetilde\ell
}\right)^{p}\int_{B_r}\!(\mu^2+|\nabla u|^2+|\nabla v|^2)^{\frac{p}{2}}\, dx\,,\end{eqnarray}
where we used  \eqref{diff444}. From \eqref{sup} and \eqref{diff401}, we  get
\begin{eqnarray}\label{diff55}
\biggl(\int_{B_\ro}\!|\nabla u|^p\,dx\biggr)^{1/p}\!\!\!&\leq&\!\!\!\biggl(\int_{B_\ro}\!|\nabla u-\nabla v|^p\,dx\biggr)^{1/p}+\biggl(\int_{B_\ro}\!|\nabla v|^p\,dx\biggr)^{1/p} \cr\cr
 \!\!\!&\leq& c_{n,p} \left(\frac{\beta L}{\widetilde\ell
}\right)\left(\int_{B_r}\!(\mu^2+|\nabla u|^2+|\nabla v|^2)^{\frac{p}{2}}\, dx\right)^{1/p}\cr\cr
&+&c_{n,p}\left(\frac{\widetilde L}{\widetilde\ell}\right)^{\frac{\sigma}{p}}\Bigl(\frac{\ro}{r}\Bigr)^{n/p}
\biggl(\int_{B_r}\!(\mu^2+|\nabla v|^2)^{\frac{p}{2}}\,dx\biggr)^{1/p}\cr\cr
&\le&c_{n,p} \left(\frac{\beta L}{\widetilde\ell
}\right)\left(\int_{B_r}\!(\mu^2+|\nabla u|^2)^{\frac{p}{2}}\, dx\right)^{1/p}\cr\cr
&+&c_{n,p}\left[\left(\frac{\beta L}{\widetilde\ell
}\right)+\left(\frac{\widetilde L}{\widetilde\ell}\right)^{\frac{\sigma}{p}}\Bigl(\frac{\ro}{r}\Bigr)^{n/p}\right]
\biggl(\int_{B_r}\!(\mu^2+|\nabla v|^2)^{\frac{p}{2}}\,dx\biggr)^{1/p}\,.
\end{eqnarray}
By virtue of \eqref{due4}, that holds for all $p>1$, from estimate \eqref{diff55}  we obtain
\begin{eqnarray}\label{diff5555}
\biggl(\int_{B_\ro}\!|\nabla u|^p\,dx\biggr)^{1/p}&\le&c_{n,p} \left(\frac{\beta L}{\widetilde\ell
}\right)\left(\int_{B_r}\!(\mu^2+|\nabla u|^2)^{\frac{p}{2}}\, dx\right)^{1/p}\cr\cr
&+&c_{n,p}\left[\left(\frac{\beta L}{\widetilde\ell
}\right)+\left(\frac{\widetilde L}{\widetilde\ell}\right)^{\frac{\sigma}{p}}\Bigl(\frac{\ro}{r}\Bigr)^{n/p}\right]
\left(\frac{\widetilde L}{\widetilde \ell}\right)^{\frac{1}{p}} \biggl(\int_{B_r}(\mu^2+|\nabla u|^2)^{\frac{p}{2}}\,dx \biggr)^{1/p}\cr\cr
 &+&c(n,p,\widetilde L,\widetilde \ell,\mu) r^\frac{n}{p}  \cr\cr
&\le&c_{n,p}\left[\left(\frac{\beta L}{\widetilde\ell
}\right)\left(\frac{\widetilde L}{\widetilde \ell}\right)^{\frac{1}{p}}+\left(\frac{\widetilde L}{\widetilde\ell}\right)^{\frac{\sigma+1}{p}}\Bigl(\frac{\ro}{r}\Bigr)^{n/p}\right]
 \biggl(\int_{B_r}(\mu^2+|\nabla u|^2)^{\frac{p}{2}}\,dx \biggr)^{1/p}\cr\cr &+&c(n,p,\widetilde L,\widetilde \ell,\mu)r^\frac{n}{p}\,.
\end{eqnarray}
 Therefore
\begin{eqnarray}\label{diff5555}
\int_{B_\ro}\!|\nabla u|^p\,dx
&\le&c_{n,p}\left[\left(\frac{\beta L}{\widetilde\ell
}\right)\left(\frac{\widetilde L}{\widetilde \ell}\right)^{\frac{1}{p}}+\left(\frac{\widetilde L}{\widetilde\ell}\right)^{\frac{\sigma+1}{p}}\Bigl(\frac{\ro}{r}\Bigr)^{n/p}\right]^p
 \int_{B_r}(\mu^2+|\nabla u|^2)^{\frac{p}{2}}\,dx \cr\cr &+&c(n,p,\widetilde L,\widetilde \ell,\mu) r^n\,.
\end{eqnarray}

\noindent Hence, both estimates \eqref{due4bis} and \eqref{diff5555}   can be written as
\begin{equation}\label{fin11}
\int_{B_\ro}\!|\nabla u|^p\,dx\le c_{n,p}\left[\zeta+\left(\frac{\widetilde L}{\widetilde\ell}\right)^{\frac{\sigma+1}{p}}\left(\frac{\rho}{r}\right)^{n/p}\right]^p\int_{B_r}\!|\nabla u|^p\,dx+c(n,p,\widetilde L,\widetilde \ell,\mu)r^n\,,
 \end{equation}
 where
 \begin{equation}\label{fin115}\zeta:=\begin{cases}
\left(\frac{\beta L}{\widetilde\ell
}\right)^{\frac{2}{p}}\qquad\qquad\qquad \mathrm{if}\,\,p\ge 2,\cr\cr
\left(\frac{\beta L}{\widetilde\ell
}\right)\left(\frac{\widetilde L}{\widetilde \ell}\right)^{\frac{1}{p}}\qquad\qquad \mathrm{if}\,\,1<p<2.
\end{cases}
\end{equation}
 We find the largest $\zeta<1$ for which there exists $\vartheta<1$ such that
\begin{equation*}
c_{n,p}\Biggl(\zeta+\left(\frac{\widetilde L}{\widetilde\ell}\right)^{\frac{\sigma+1}{p}}\vartheta^{n/p}\Biggr)^p=\vartheta^{n-1}\,.
\end{equation*}
This equality is equivalent to
\begin{equation*}
\zeta=\frac{\vartheta^{(n-1)/p}}{c^{\frac{1}{p}}}-\left(\frac{\widetilde L}{\widetilde\ell}\right)^{\frac{\sigma+1}{p}}\vartheta^{n/p}=:\,f(\vartheta),
\end{equation*}
where, for simplicity, we set $c=c_{n,p}$, $c>1$.
Note that such $\vartheta,\zeta\in[0,1)$ exist. Indeed
$$\frac{df}{d\vartheta}(\vartheta)=\frac{1}{p}\vartheta^{\frac{n}{p}-1}\left(\frac{n-1}{c^{\frac{1}{p}}}\vartheta^{-\frac{1}{p}}-n\left(\frac{\widetilde L}{\widetilde\ell}\right)^{\frac{\sigma+1}{p}}\right)\,,$$
and so
$$\frac{df}{d\vartheta}(\vartheta)=0\qquad\Leftrightarrow\qquad \vartheta=\frac{1}{c}\left(\frac{n-1}{n}\right)^{p}\left(\frac{\widetilde L}{\widetilde\ell}\right)^{-(\sigma+1)}.$$
Set $$\vartheta_0:=\frac{1}{c}\left(\frac{n-1}{n}\right)^{p}\left(\frac{\widetilde L}{\widetilde\ell}\right)^{-(\sigma+1)}\qquad \mathrm{and}\qquad \zeta_0:=f(\vartheta_0)\,.$$
Since $\widetilde\ell\le \widetilde L$ and $c\ge 1$ it follows that $\vartheta_0\in (0,1)$ and
$$f(\vartheta_0)=\max_{\vartheta\in [0,1]}f(\vartheta).$$
Moreover,
\begin{eqnarray*}
\zeta_{_0}&=&\frac{\vartheta_0^{(n-1)p}}{c^{\frac{1}{p}}}-\left(\frac{\widetilde L}{\widetilde\ell}\right)^{\frac{\sigma+1}{p}}\vartheta_0^{\frac{n}{p}}=
\vartheta_0^{\frac{n}{p}}\Bigg(\frac{\vartheta_0^{-\frac{1}{p}}}{c^{\frac{1}{p}}}-\left(\frac{\widetilde L}{\widetilde\ell}\right)^{\frac{\sigma+1}{p}}\Bigg)\cr\cr
&=&\vartheta_0^{\frac{n}{p}}\Bigg(\frac{n}{n-1}-1\Bigg)\left(\frac{\widetilde L}{\widetilde\ell}\right)^{\frac{\sigma+1}{p}}=\vartheta_0^{\frac{n}{p}}\Bigg(\frac{1}{n-1}\Bigg)\left(\frac{\widetilde L}{\widetilde\ell}\right)^{\frac{\sigma+1}{p}}
\end{eqnarray*}
We write
$$\zeta_{_0}=\frac{(n-1)^{(n-1)}}{n^n}\frac{1}{c^{\frac{n}{p}}}\left(\frac{\widetilde L}{\widetilde\ell}\right)^{\frac{\sigma+1}{p}(1-n)}=\tilde c_{n,p}\left(\frac{\widetilde \ell}{\widetilde L}\right)^{\widetilde \sigma}$$
with $\displaystyle{\tilde c_{n,p}:=\frac{(n-1)^{(n-1)}}{n^n}\frac{1}{c^{\frac{n}{p}}}}$ and $\widetilde \sigma:=\frac{\sigma+1}{p}(n-1)$. Note that $\zeta_0\in (0,1)$.
 In case in which $p\ge 2$, we need
\begin{equation}\label{due8000}
\Bigl(\frac{\beta L}{\widetilde\ell}\Bigr)^{\frac{2}{p}}<\zeta_{_0}\,.
\end{equation}
Recalling that $\widetilde \ell\ge (\alpha+1)\ell$ and $\widetilde L\le (\beta+1)L$, in order to have \eqref{due8000} it suffices to impose that
\begin{equation}\label{due7bis}
\left(\frac{\beta}{\alpha+1}\right)^{\frac{2}{p}}\left(\frac{\beta+1}{\alpha+1}\right)^{\widetilde\sigma}< \tilde c_{n,p}\left(\frac{\ell}{L}\right)^{\tilde\sigma+\frac{2}{p}}.
\end{equation}
In fact
$$(\beta L)^\frac{2}{p}\Big((\beta+1)L\Big)^{\widetilde\sigma}<\tilde c_{n,p}\Big((\alpha+1)\ell)^{\widetilde \sigma+\frac{2}{p}}\quad\Rightarrow\quad(\beta L)^\frac{2}{p}\widetilde L^{\widetilde \sigma}<\tilde c_{n,p}(\widetilde\ell)^{\,\widetilde \sigma+\frac{2}{p}}$$
$$\quad\Leftrightarrow\quad \Bigl(\frac{\beta L}{\widetilde\ell}\Bigr)^{\frac{2}{p}}<\zeta_{_0}=\tilde c_{n,p}\left(\frac{\widetilde \ell}{\widetilde L}\right)^{\tilde \sigma}, $$
and inequality \eqref{due7bis} is clearly fulfilled if the ratio $\frac{\beta}{\alpha+1}$ is sufficiently small.

\noindent Similarly, in case in which $1<p<2$, we need
\begin{equation}\label{due88}
\left(\frac{\beta L}{\widetilde\ell
}\right)\left(\frac{\widetilde L}{\widetilde \ell}\right)^{\frac{1}{p}}<\zeta_{_0}\,.
\end{equation}
In order to have \eqref{due88}, it suffices to impose
\begin{equation}\label{due888}
\frac{\beta}{\alpha+1}\left(\frac{\beta+1}{\alpha+1}\right)^{\tilde\sigma+\frac{1}{p}}<\tilde c_{n,p}\left(\frac{\ell}{L}\right)^{\tilde\sigma+1+\frac{1}{p}}.
\end{equation}
In fact
$$\beta (\beta+1)^{\widetilde\sigma+\frac{1}{p}}L^{\tilde\sigma+1+\frac{1}{p}}<\widetilde c_{n,p}((\alpha+1)\ell)^{\widetilde\sigma+1+\frac{1}{p}}\quad\Rightarrow\quad
(\beta L)\widetilde L^{\widetilde\sigma+\frac{1}{p}}<\tilde c_{n,p}(\widetilde \ell)^{\widetilde\sigma+1+\frac{1}{p}}$$
$$\quad\Leftrightarrow\quad\left(\frac{\beta L}{\widetilde\ell
}\right)\left(\frac{\widetilde L}{\widetilde \ell}\right)^{\frac{1}{p}}<\tilde c_{n,p}\left(\frac{\widetilde \ell}{\widetilde L}\right)^{\widetilde \sigma}=\zeta_{_0} $$
and inequality \eqref{due888} is clearly fulfilled if the ratio $\frac{\beta}{\alpha+1}$ is sufficiently small.

\noindent Then, choosing $\alpha,\beta$ such that \eqref{due8000} (if $p\ge 2$) or \eqref{due88} (if $1<p< 2$) are satisfied,  in view of \eqref{fin11}   there exist $\vartheta\in(0,\vartheta_0)$ and $\bar\delta>0$, depending on $\alpha,\beta,n,p,\ell,L$, such that
$$
\int_{B_{\vartheta r}}\!|\nabla u|^p\,dx\leq\vartheta^{n-1+p\bar\delta}\int_{B_r}\!|\nabla u|^p\,dx+c(n,p,\mu, \widetilde L, \widetilde \ell)r^n\,.
$$
Since $r<1$, the term $r^n$ can be majorized by $r^{n-1+p\delta}$, for every $0<\delta<\min\{\bar \delta, \frac{1}{p}\}$, and from the previous estimate, we deduce that
$$
\int_{B_{\vartheta r}}\!|\nabla u|^p\,dx\leq\vartheta^{n-1+p\bar\delta}\int_{B_r}\!|\nabla u|^p\,dx+c(n,p,\mu, \widetilde L, \widetilde \ell)r^{n-1+p\delta}\,.
$$
This estimate, by virtue of Lemma \ref{iter},  yields that for all $0<\ro<r<1$
\begin{equation}\label{hold}
\int_{B_\ro}\!|\nabla u|^p\,dx\leq c(n)\Bigl(\frac{\ro}{r}\Bigr)^{n-1+p\delta}\int_{B_r}\!|\nabla u|^p\,dx+c\rho^{n-1+p\delta}
\end{equation}
So, for $0<\ro\ll 1$
\begin{equation*}
\int_{B_\ro}\!|\nabla u|^p\,dx\leq c\rho^{n-1+p\delta}
\end{equation*}
and by H\"older's inequality
\begin{equation*}\label{hold1}
\int_{B_\ro}\!|\nabla u|\,dx\le c \left(\int_{B_\ro}\!|\nabla u|^p\,dx\right)^{\frac{1}{p}}\ro^{\frac{n}{p'}}\le c \ro^{\frac{n-1+p\delta+n(p-1)}{p}}=c\rho^{n-1+\delta+\frac{1}{p'}}\,.
\end{equation*}
By  Theorem \ref{Morrey}, the previous inequality  implies that $u\in C^{0,\frac{1}{p'}+\delta}_{\mathrm{loc}}(\Om)$  whenever \eqref{due7bis} (if $p\ge 2$) or \eqref{due888} (if $1<p<2$)
hold true.
\medskip

\noindent {\bf Step 2.} Fix a point $x\in\Om$ and let $\bar r>0$ be such that dist$(x,\partial\Om)>\bar r$. Consider  $0<r<r_0\le \bar r$ and denote by $A$ any set of finite perimeter such that $E\Delta A\subset\!\subset B_r(x)$. From Theorem~\ref{uno} we have that
$$
{\mathcal I}_{\lambda_0}(u,E)\leq{\mathcal I}_{\lambda_0}(u,A)\,,
$$
and thus
$$ \int_\Om \Big(F(\nabla u)+\chi_{_{E}}G(\nabla u)\Big)\,dx+P(E,\Om)+\lambda_0\big||E|-d\big|$$
$$\le \int_\Om \Big(F(\nabla u)+\chi_{_{A}}G(\nabla u)\Big)\,dx+P(A,\Om)+\lambda_0\big||A|-d\big|. $$
Using that $E\Delta A\subset\!\subset B_r(x)$,  we deduce that
\begin{eqnarray*}
P(E,B_r(x))-P(A,B_r(x))&\leq&\int_{B_r}\Big(\chi_{_A}(x)-\chi_{_E}(x)\Big)G(\nabla u)\,dx+\lambda_0\bigl||A|-|E|\bigr|
\cr\cr &\leq& \beta L\int_{B_r}|\nabla u|^p+cr^n\,,
\end{eqnarray*}
where we invoked assumption (G2). By the decay estimate  \eqref{hold}, we infer that
\begin{equation*}
P(E,B_r(x))-P(A,B_r(x))\leq cr^{n-1+p\delta}+c r^n\le cr^{n-1+p\delta}
\end{equation*}
since  $r<1$.
As $\delta$  can be replaced  by any smaller number,  we can choose $p\delta<\frac{1}{2}$ and the result  follows  from Theorem~\ref{tre}.
\end{proof}
\vskip 5pt
\noindent

\bigskip

\section{Partial regularity--Proof of Theorem \ref{duetre}}

\bigskip
In this section, we prove that a  partial regularity result holds without imposing any bounds on $\alpha$ and $\beta$,
as stated in Theorem~\ref{duetre}.

\bigskip

\begin{proof}[Proof of Theorem~\ref{duetre}]

 Set
$$\Omega_0:=\left\{x\in \Omega: \limsup_{\rho\to 0}\frac{1}{{\rho}^{n-1+p\delta}}\int_{B_{\rho}}|\nabla u|^p=0\,\right\},$$
for an arbitrary $0<\delta<\frac{1}{p}$.
Note that (see Theorem 3, Section 2.4.3,  in \cite{EG}), $|\Omega\setminus \Omega_0|=0\,.$
Fix  a point $x\in\Om_0$  and let $r_0$ be such that dist$(x,\partial\Om)>r_0$.
Since $x\in\Om_0$, for every $\varepsilon>0$ there exists a radius $ R\,=\,R(\varepsilon)<r_0$ such that
\begin{equation}\label{epsi}\frac{1}{{ {r}}^{n-1+p\delta}}\int_{B_{r} }|\nabla u|^p<\varepsilon\end{equation}
for all $ 0<r\le R(\varepsilon)$.
By  \eqref{fin11} and \eqref{due8000},   for all $0<\ro<r$ we have
\begin{equation*}
\int_{B_\ro}\!|\nabla u|^p\,dx\le c_{n,p}\left[\zeta+\left(\frac{\widetilde L}{\widetilde\ell}\right)^{\frac{\sigma+1}{p}}\left(\frac{\rho}{r}\right)^{n/p}\right]^p\int_{B_r}\!|\nabla u|^p\,dx+c(n,p,\widetilde L,\widetilde \ell,\mu)r^n
 \end{equation*}
Inserting \eqref{epsi} in previous inequality, we get
\begin{equation*}
\int_{B_\ro}\!|\nabla u|^p\,dx\leq c\,\, \varepsilon^{\frac{1}{2}}\,\,\Biggl[1+\Bigl(\frac{\ro}{r}\Bigr)^{n/p}
\Biggr]^p r^{\frac{n-1+p\delta}{2}}\left(\int_{B_r}\!|\nabla u|^p\,dx\right)^{\frac{1}{2}}+cr^n\,,
\end{equation*}
where $c=c(n,p,\mu, \alpha, \beta,\ell, L)$.
By Young's inequality, we deduce that
\begin{eqnarray}\label{due511}
\int_{B_\ro}\!|\nabla u|^p\,dx&\leq& c\,\, \varepsilon^{\frac{1}{2}}\,\,\Biggl[1+\Bigl(\frac{\ro}{r}\Bigr)^{n/p}
\Biggr]^p\left\{\int_{B_r}\!|\nabla u|^p\,dx+r^{n-1+p\delta}\right\}+cr^n\cr\cr
&\leq& c\,\, \varepsilon^{\frac{1}{2}}\,\,\Biggl[1+\Bigl(\frac{\ro}{r}\Bigr)^{n/p}
\Biggr]^p\int_{B_r}\!|\nabla u|^p\,dx+cr^{n-1+p\delta}\,.
\end{eqnarray}
for every $0<\rho<r\le R(\varepsilon)$, since we may suppose, without loss of generality, that $r<1$. Therefore, in particular, writing \eqref{due511} for $\rho=\frac{r}{2}$, we get
\begin{equation}\label{5110}
\int_{B_{\frac{r}{2}}}\!|\nabla u|^p\,dx\leq
 c\,\, \varepsilon^{\frac{1}{2}}\,\,\Biggl[1+\Bigl(\frac{1}{2}\Bigr)^{n/p}
\Biggr]^p\int_{B_r}\!|\nabla u|^p\,dx+cr^{n-1+p\delta}\,.
\end{equation}
Choosing   $\varepsilon$ in \eqref{5110} such that
\begin{equation*}
\varepsilon^{\frac{1}{2}}<\frac{2^{1-p}} {c\Big(1+2^{\frac{n}{p}}\Big)^{p}}\,,
\end{equation*} we obtain
\begin{eqnarray}\label{due81}
\int_{B_{\frac{r}{2}}}\!|\nabla u|^p\,dx&\leq&
 \frac{2^{1-p}} {\Big(1+2^{\frac{n}{p}}\Big)^{p}}\Biggl[1+\Bigl(\frac{1}{2}\Bigr)^{n/p}
\Biggr]^p\int_{B_r}\!|\nabla u|^p\,dx+cr^{n-1+p\delta}\cr\cr
&=& \frac{2^{1-p}} {\Big(1+2^{\frac{n}{p}}\Big)^{p}}\frac{\Big(1+2^{\frac{n}{p}}\Big)^{p}}{2^n}\int_{B_r}\!|\nabla u|^p\,dx+cr^{n-1+p\delta}\cr\cr
&=&\left(\frac{1}{2}\right)^{n-1+p}\int_{B_r}\!|\nabla u|^p\,dx+cr^{n-1+p\delta}\,.
\end{eqnarray}
From \eqref{due81}, thanks to Lemma \ref{iter} applied with $\varphi(r):=\int_{B_r}|\nabla u|^p\,dx$ and $\vartheta=\frac{1}{2}$, we obtain that
\begin{equation}\label{due82}
\int_{B_\ro}\!|\nabla u|^p\,dx\leq c\rho^{n-1+p\delta}
\end{equation}
for all $0<\ro<r<r_0$ and some $c=c(n,p, \alpha, \beta,\ell, L)$.
Hence, by virtue of Theorem \ref{Morrey} and H\"older inequality, we deduce that $u\in C^{0,\frac{1}{p'}+\delta}(\Omega_0)$, for every $0<\delta<\frac{1}{p}$.

\noindent Let us denote by $A$ any set of finite perimeter such that $E\Delta A\subset\!\subset B_{\rho}(x)$. From Theorem~\ref{uno} we have that
$$
{\mathcal I}_{\lambda_0}(u,E)\leq{\mathcal I}_{\lambda_0}(u,A)\,,
$$
therefore, by assumption (G1) and the decay estimate \eqref{due82}, we deduce that
\begin{eqnarray*}
P(E,B_{\rho}(x))-P(A,B_{\rho}(x))&\leq&\int_{B_{\rho}}(\chi_{A}(x)-\chi_{E}(x))G(\nabla u)\,dx+\lambda_0\bigl||A|-|E|\bigr|\cr\cr
&\leq& \beta L\int_{B_{\rho}}|\nabla u|^p\,dx+\lambda_0\bigl||A|-|E|\bigr|
\leq c  \rho^{n-1+p\delta}+c\lambda_0{\rho}^n\,
\end{eqnarray*}
and the conclusion follows again by Theorem ~\ref{tre} applied to  $\Omega_0$ in place of $\Omega$.

\end{proof}\bigskip

\noindent
{\bf Acknowledgment.} The authors warmly thank the Center for Nonlinear Analysis (NSF Grants No. DMS-
0405343 and DMS-0635983), where part of this research was carried out. The research of
I. Fonseca was partially funded by the National Science Foundation under Grant No. DMS-
0905778.

\noindent
Dipartimento di Ingegneria - Universit\`{a} del Sannio,  82100 Benevento, Italy

\noindent
{\em E-mail address}: carozza@unisannio.it
\medskip

\noindent
Department of Mathematical Sciences , Carnegie Mellon University, Pittsburgh PA

\noindent
{\em E-mail address}: fonseca@andrew.cmu.edu
\medskip

\noindent
Universit\`{a} di Napoli ''Federico II'' Dipartimento di Mat.~e Appl. 'R.~Caccioppoli',
Via Cintia, 80126 Napoli, Italy

\noindent
{\em E-mail address}: antpassa@unina.it


\begin{thebibliography}{99}
\bibitem{AC} H.W. Alt \& L.A.Caffarelli, {\em Existence and regularity results for a minimum problem with free boundary}. J. Reine Angew. Math. {\bf 325}, 107--144 (1981).
\bibitem{AF} E. Acerbi \& N. Fusco, {\em Regularity forminimizers of non-quadratic functionals: the case
$1 < p < 2$.} J. Math. Anal. Appl. {\bf 140}, 115–-135, (1989).
\bibitem{af1} E. Acerbi \& N. Fusco, {\em A regularity theorem for minimizers of quasi-convex integrals}, Arch. Rat. Mech. Anal {\bf 99}, 261–-281,  (1987)
\bibitem{AB} L.Ambrosio \& G.Buttazzo, {\em An optimal design problem with perimeter penalization}. Calc. Var. Part. Diff. Eq. {\bf 1}, 55--69 (1993).

\bibitem{AFP}{ L.Ambrosio, N.Fusco \& D.Pallara}, {\em Functions of Bounded Variation and Free Discontinuity Problems}. Oxford University Press (2000).
\bibitem{Bom} {E.Bombieri}, {\em Regularity theory for almost minimal currents}. Arch. Rational Mech. Anal. {\bf 78}, 99--130 (1982).

\bibitem{cp} {M. Carozza \& A. Passarelli Di Napoli},{\em A regularity theorem for minimisers of quasiconvex integrals: The case $1< p< 2$} Proc. of the Royal Society of Edinburgh-A-Math.
{\bf 126, 6} 1181--1200 (1996)
\bibitem{EF}{L.~Esposito and N.~Fusco},
{\em A remark on a free interface problem with volume constraint}.
J.~Convex~Anal. {\bf 18} (2011), no. 2, 417-–426.
\bibitem{EG}{L.C.~Evans \& F.R.~Gariepy}, {\em Measure theory and fine properties of functions.} Studies in Advanced Mathematics. CRC Press, Boca Raton, FL (1992).
\bibitem{ff}{I.~Fonseca \& N.~Fusco},
{\em Regularity results for anisotropic image segmentation models}.
Ann.~Sc.~Norm.~Super.~Pisa {\bf 24} (1997), 463--499.
\bibitem{FFLM}{I.Fonseca, N.Fusco, G.Leoni \& V. Millot},
{\em Material voids in elastic solids with anisotropic surface energies}.
J. Math. Pures Appl.{ \bf(9) 96} (2011), no. 6,
\bibitem{FFLMo} {I.Fonseca, N.Fusco, G.Leoni \& M.Morini}, {\em Equilibrium configurations of epitaxially strained crystalline films:existence and regularity results}. Arch. Rational Mech. Anal. {\bf 186}, 477--537 (2007).
\bibitem{fh}  {N.  Fusco \& J. Hutchinson}, {\em $C^{1,\alpha}$ partial regularity of functions minimising quasiconvex
integrals}, Manuscripta Math. 54 (1985), 121--143.
\bibitem{Gia} M. Giaquinta, {\em Multiple integrals
in the  calculus of  variations and nonlinear
ellyptic systems.} Annals of Mathematical Studies. Princeton University Press (1983).
\bibitem{GM} M. Giaquinta \& G. Modica, {\em Partial regularity of minimizers of quasiconvex integrals.}
Ann. Inst. Henri Poincaré, Anal. Non Linéaire {\bf 3}, 185–-208, (1986).

    \bibitem{GTr} {D. Gilbarg \& N.S. Trudinger}, {\em  Elliptic partial differential equations of second order}. Second edition. Grundlehren der Mathematischen Wissenschaften , 224. Springer-Verlag, Berlin, (1983).

\bibitem{gi}{E.~Giusti}. {\em Direct methods in the calculus of variations}.
World Scientific, 2003.
\bibitem{Gur} {M.Gurtin}, {\em On phase transitions with bulk, interfacial, and boundary enwergy}. Arch. Rational Mech. Anal. {\bf 96}, 243--264 (1986).
\bibitem{Lar} {C.J.Larsen}, {\em  Regularity of componrnts in optimal design problems with perimeter penalization}. Calc. Var. Part. Diff. Eq. {\bf 16}, 17--29 (2003).
\bibitem{Lin} {F.H.Lin}, {\em Variational problems with free interfaces}. Calc. Var. Part. Diff. Eq. {\bf 1}, 149--168 (1993).
\bibitem{KL} {F.H.Lin \&  R.V. Kohn}, {\em Partial regularity for optimal design problems involving both bulk and surface energies}. Chin. Ann. of Math.  {\bf 20B,2}, 137--158 (1999)
    \bibitem{sy}{V.~\v{S}ver\'{a}k and X.~Yan}.
{\em Non-Lipschitz minimizers of smooth uniformly convex variational integrals}.
Proc.~Nat.~Acad.~Sci. USA {\bf 99} (2002), 15269--15276.
\bibitem{Tam} {I.Tamanini }, {\em Boundaries of Caccioppoli sets with H\"older-continuous normal vector}.J. Reine Angew. Math. {\bf 334}, 27--39 (1982)


\end{thebibliography}
\end{document}